\documentclass[a4paper,reqno]{amsart}

\usepackage{amsmath, amssymb, amsthm}
\usepackage[english]{babel}
\usepackage{enumerate}
\usepackage[all]{xy}
\usepackage{graphicx}
\usepackage{mathtools}
\usepackage{mathrsfs}
\usepackage{hyperref}
\usepackage{verbatim}
\usepackage{color}
\usepackage{tikz-cd}
\tikzcdset{arrow style=tikz, diagrams={>=stealth}}

\allowdisplaybreaks
\newcommand*{\id}{\textup{id}}

\numberwithin{equation}{section}

\theoremstyle{plain}

\newtheorem{thm}{Theorem}[section]

\newtheorem{prop}[thm]{Proposition}
 \newtheorem{cor}[thm]{Corollary}
\newtheorem{defi}[thm]{Definition}

\theoremstyle{remark}

\newtheorem{rem}[thm]{Remark}

\numberwithin{equation}{section}

\newcommand{\ot}{\otimes}

\newcommand{\beq}{\begin{equation}}
\newcommand{\eeq}{\end{equation}}

\newcommand{\A}{\mathcal{A}}

\newcommand{\cL}{\mathcal{L}}
\newcommand{\cR}{\mathcal{R}}

\newcommand{\one}[1]{{#1}{}_{\scriptscriptstyle{(1)}}}
\newcommand{\two}[1]{{#1}{}_{\scriptscriptstyle{(2)}}}

\newcommand{\p}{{}_{\scriptscriptstyle{\hat{+}}}}
\newcommand{\np}{{}_{\scriptscriptstyle{\hat{[+]}}}}

\newcommand{\m}{{}_{\scriptscriptstyle{\hat{-}}}}
\newcommand{\nm}{{}_{\scriptscriptstyle{\hat{[-]}}}}

\newcommand{\lbiprod}{{>\!\!\!\triangleleft\kern-.33em\cdot}}
\newcommand{\rbiprod}{{\cdot\kern-.33em\triangleright\!\!\!<}}

\newcommand{\z}{{}_{\scriptscriptstyle{(0)}}}
\newcommand{\rz}{{}_{\scriptscriptstyle{[0]}}}
\renewcommand{\o}{{}_{\scriptscriptstyle{(1)}}}
\newcommand{\ro}{{}_{\scriptscriptstyle{[1]}}}
\newcommand{\mo}{{}_{\scriptscriptstyle{(-1)}}}
\newcommand{\rmo}{{}_{\scriptscriptstyle{[-1]}}}
\renewcommand{\t}{{}_{\scriptscriptstyle{(2)}}}
\newcommand{\rt}{{}_{\scriptscriptstyle{[2]}}}
\newcommand{\mt}{{}_{\scriptscriptstyle{(-2)}}}

\renewcommand{\th}{{}_{\scriptscriptstyle{(3)}}}
\newcommand{\rth}{{}_{\scriptscriptstyle{[3]}}}

\newcommand{\la}{{\triangleright}}
\newcommand{\ra}{{\triangleleft}}
\newcommand{\bla}{{\blacktriangleright}}
\newcommand{\bra}{{\blacktriangleleft}}

\DeclareMathOperator{\tens}{\otimes}

\newcommand{\cH}{\mathcal{H}}
\newcommand{\CM}{\mathcal{M}}

\begin{document}

\author{Xiao Han}
\address{Queen Mary University of London\\
		School of Mathematical Sciences, Mile End Rd, London E1 4NS, UK}

\email{x.h.han@qmul.ac.uk}

\keywords{Hopf algebroid, bialgebroid, quantum group, Hopf bimodule, Yetter-Drinfeld module}

\title{Hopf bimodules and Yetter-Drinfeld modules of Hopf algebroids}
%

\begin{abstract}
 We construct Hopf bimodules and Yetter-Drinfeld modules of Hopf algebroids as a generalization of the theory for Hopf algebras. More precisely, we show that the categories of Hopf bimodules and Yetter-Drinfeld modules over a Hopf algebroid are equivalent (pre-)braided monoidal categories. Moreover, we also study the duality between finitely generated projective Yetter-Drinfeld modules.
 \end{abstract}

\maketitle

\section{Introduction}

The concept of Hopf bimodules (bicovariant bimodules) of  Hopf algebras is introduced by Woronowicz in \cite{SLW}, which aims to construct a quantum Lie algebra with a braided Lie bracket. The 4 variations of
 Yetter-Drinfeld modules of a bialgebra are studied by Radford and Towber in \cite{RT}. Moreover, Schauenburg showed in \cite{schau2} that the categories of Yetter-Drinfeld modules and Hopf bimodules of a Hopf algebra with a bijective antipode are equivalent braided monoidal categories.  In recent years, the theory of groupoids and Lie algebroids has been very well established and has a lot of applications ranging from pure math to theoretical physics. Moreover, as a quantization of groupoids, quantum groupoids are studied initially by J-H. Lu \cite{Lu}, P. Xu \cite{Xu}, and others, which are also very important because they provide powerful mathematical tools for quantum algebra.   In this article, we are going to generalize the theory of Hopf bimodules and Yetter-Drinfeld modules to  Hopf algebroids with bijective antipodes\cite{BS}.

The concept of bialgebroids is quite clear, but to introduce the `Hopf' structure, one has to deal with something to play the role of an  `antipode'. It is first defined by Schauenburg in \cite{schau1}, that a left Hopf algebroid is a left bialgebroid which is a left Galois object of itself. Moreover, the category of left-left Yetter-Drinfeld modules of a left bialgebroid $\cL$ is equivalent to the category of weak centers of the left modules of $\cL$. There is also a well-known Hopf algebroid introduced by B\"ohm and  Szlach\'anyi in \cite{BS}, which required to have a bijective antialgebra map for the antipode. This is much stronger than a left Hopf algebroid. In particular, a Hopf algebroid is not only a left Hopf algebroid but also an anti-left Hopf algebroid, (anti-)right Hopf algebroid. In this paper, the first main result is that these properties allow us to construct braiding maps that make the categories of Hopf bimodules and Yetter-Drinfeld modules of Hopf algebroids (pre-)braided monoidal categories. Moreover, we also show that these two braided monoidal categories are equivalent. The second main result is to show the right dual of a right-right Yetter-Drinfeld module of a left finite Hopf algebroid $\cR$ is a left-right Yetter-Drinfeld module of the right dual Hopf algebroid $\cR^{\vee}$. Moreover, the right dual functor is shown to be a braided monoidal functor which makes these two categories equivalent. For examples of Hopf algebroids, we note the special quantum gauge groupoids associated with the monopole and instanton bundles (cf. \cite{HG}, \cite{HGY}).

We also note that due to the important role of Hopf bimodules in the construction of differential structures on  Hopf algebras \cite{SLW}, one can expect that our results here should also be important for the construction of differential structures on Hopf algebroids and associated quantum Lie algebroids. This application will be looked at elsewhere and should in particular allow the construction of a quantum Atiyah Lie-algebroid associated to a quantum gauge groupoid $\cL(P, H)$ (Ehresmann--Schauenburg Hopf algebroid) introduced in \cite{schau2.5}. The latter is known to be a left Hopf algebroid in \cite{HM} and also to admit a weaker antipode under some conditions.

The paper consists of 4 sections. Section \ref{sec2} is the algebraic preliminaries that gives the basic definitions that we need. In section \ref{sec3} and \ref{sec4}, we introduce Hopf bimodules of Hopf algebroids and show that the monoidal categories of Hopf bimodules and Yetter-Drinfeld modules are equivalent as braided monoidal categories. In section \ref{sec5}, we first introduce the dual pairing between right-right Yetter-Drinfeld modules and left-right Yetter-Drinfeld modules. We then show that the right dual functor is a braided monoidal functor. In the appendix \ref{sec6}, we collect together some key identities for working with Hopf algebroids which are used in the main text. Some of these are standard, while the last group, for (anti-)right Hopf algebroids, is new.

It is also useful to compare this work with others. In conclusion, this work mainly focuses on Yetter-Drinfeld modules and Hopf bimodules of Hopf algebroids with bijective antipode (in the sense of \cite{BS}). Therefore, sections 3 and 4 are the `-oid' generalization of \cite{schau2}. There are many other works with the related topics. In \cite{BK},  B\'alint,  Szlach\'anyi have shown that the category of two-sided Yetter-Drinfeld modules is a braided monoidal category and equivalent to the centers of the modules of a Frobenius Hopf algebroid. Therefore, it could be viewed as a special case of the work in [Proposition 4.4 \cite{schau1}], which has shown that the category of Yetter-Drinfeld modules of a bialgebroid is equivalent to the weak centers of the modules of the bialgebroid. Another related work is given by Kowalzig\cite{NK}, where the left-right Yetter-Drinfeld modules and their relation to the centers of the corresponding bimodule category are studied. Moreover, this work also plays an important role in the cyclic (co)homology theory. In \cite{NK}, however, the corresponding bialgebroids are left and anti-left Hopf algebroids (therefore without antipode). Although this paper deals with Hopf algebroids at a precise level, there is another work\cite{RAA}, where R. Alcal\'a has dealt with Hopf algebroids and their modules in a more categorical way.

\subsection*{Acknowledgments} I thank Shahn Majid and Peter Schauenburg for helpful discussions. I was supported by the European Union's Horizon 2020 research and innovation program under the Marie Sk\l odowska-Curie grant agreement No 101027463.

\section{Basic algebraic preliminaries} \label{sec2}

\subsection{Left and right bialgebroids and its modules}

Here we recall the basic definitions (cf. \cite{BW}, \cite{Boehm}). Let $B$ be a unital algebra over a field $k$.
A {\em $B$-ring}  means a unital algebra in the monoidal category ${}_B\CM_B$ of $B$-bimodules. Likewise,  a  {\em $B$-coring} is a coalgebra in ${}_B\CM_B$. Morphisms of $B$-(co)rings are maps that preserve the (co)algebra structures and in the category ${}_B\CM_B$.

In practice, specifying a unital $B$-ring $\cL$ is equivalent to specifying a unital algebra $\cL$ (over $k$) and an algebra map $\eta:B\to \cL$. Left and right multiplication in $\cL$ pull back to left and right $B$-actions as a bimodule (so $b.X.c=\eta(b)X\eta(c)$ for all $b,c\in B$ and $X\in \cL$) and the product descends to the product $\mu_B:\cL\tens_B\cL\to \cL$ with $\eta$ the unit map. Conversely, given
$\mu_B$ we can pull back to an associative product on $\cL$ with unit $\eta(1)$.

Now suppose that $s:B\to \cL$ and $t:B^{op}\to \cL$ are algebra maps with images that commute. Then $\eta(b\tens c)=s(b)t(c)$ is an algebra map $\eta: B^e\to \cL$, where $B^e=B\tens B^{op}$, and is equivalent to making $\cL$ a $B^e$-ring. The left $B^e$-action part of this is equivalent to a $B$-bimodule structure
\begin{equation}\label{eq:rbgd.bimod}
b.X.c= s(b) t(c)X
\end{equation}
for all $b,c\in B$ and $X\in \cL$. Similarly, the right $B^{e}$-action part of this is equivalent to another $B$-bimodule structure
\begin{equation}\label{eq:rbgd.bimod1}
c^{.}X^{.}b= Xs(b) t(c).
\end{equation}

\begin{defi}\label{def:left.bgd} Let $B$ be a unital algebra. A left bialgebroid over $B$ (or left $B$-bialgebroid) is an algebra $\cL$ with commute algebra maps (`source' and `target' maps) $s_{L}:B\to \cL$ and $t_{L}:B^{op}\to \cL$ (and hence a $B^e$-ring) and a $B$-coring for the bimodule structure (\ref{eq:rbgd.bimod}) which is compatible in the sense
\begin{itemize}
\item[(i)] The coproduct $\Delta_{L}$ corestricts to an algebra map  $\cL\to \cL\times_{B} \cL$ where
\begin{equation*} \cL\times_{B} \cL :=\{\ \sum_i X_i \ot_{B} Y_i\ |\ \sum_i X_it_{L}(a) \ot_{B} Y_i=
\sum_i X_i \ot_{B}  Y_i  s_{L}(a),\quad \forall a\in B\ \}\subseteq \cL\tens_{B}\cL,
\end{equation*}
 is an algebra via factorwise multiplication.
\item[(ii)] The counit $\varepsilon$ is a left character in the following sense:
\begin{equation*}\varepsilon(1_{\cL})=1_{B},\quad \varepsilon( s_{L}(a)X)=a\varepsilon(X), \quad \varepsilon(Xs_{L}(\varepsilon(Y)))=\varepsilon(XY)=\varepsilon(Xt_{L} (\varepsilon(Y)))\end{equation*}
for all $X,Y\in \cL$ and $a\in B$.
\end{itemize}
\end{defi}

Morphisms between left $B$-bialgebroids are $B$-coring maps which are also algebra maps. In the following, we will use the sumless Sweedler indexes to denote the coproducts of left bialgebroids, namely,
 $\Delta_{L}(X)=X\o\ot_{B}X\t$. Whenever we mention a left bialgebroid (in case we haven't mentioned its base algebra), by default, it is a left bialgebroid over $B$.

By \cite{schau1}, we have the definition of left comodules of left algebroids.

\begin{defi}
    Let $\cL$ be a left bialgebroid over $B$, a left $\cL$-comodule is a $B$-bimodule $\Gamma$ (with $B$-bimodule structure denoted by $b.\rho.b'$, for all $b, b'\in B$), together with a $B$-bimodule map ${}_{L}\delta: \Gamma\to \cL\times_{B}\Gamma\subseteq\cL\ot_{B}\Gamma$, written ${}_{L}\delta(\rho)=\rho\mo\ot_{B}\rho\z$, such that
    \[(\Delta_{L}\ot_{B}\id_{\Gamma})\circ {}_{L}\delta=(\id_{\cL}\ot_{B}{}_{L}\delta)\circ {}_{L}\delta,\qquad (\varepsilon_{L}\ot_{B}\id_{\Gamma})\circ {}_{L}\delta=\id_{\Gamma},\]

where
\begin{equation*} \cL\times_{B} \Gamma :=\{\ \sum_i X_i \ot_{B} \rho_i\in \cL\tens_B\Gamma\ |\ \sum_i X_it_{L}(b) \ot_{B} \rho_i=
\sum_i X_i \ot_{B}  \rho_i  .b,\quad \forall b\in B\ \}.
\end{equation*}
\end{defi}
Here ${}_{L}\delta$ is a $B$-bimodule map in the sense that
${}_{L}\delta(b.\rho.b')=s_{L}(b)\rho\mo s_{L}(b')\ot_{B}\rho\z$.

\begin{rem}
Let $\cL$ be a left $B$-bialgebroid, thus it is a $B$-coring. It is given in [Section 1.4 of \cite{HHP}] that the comodule of $\cL$ as a $B$-coring coincides with the notion of comodule of $\cL$ as a left bialgebroid given above. More precisely, by the definition in \cite{HHP}, the image of the coaction of a comodule $\Gamma$ of a bialgebroid $\cL$ (as a $B$-coring) automatically lies in the Takeuchi product $\cL\times_{B}\Gamma$. In the following context, we will always use the definition above for more practical computation.
\end{rem}

\begin{defi}\label{def. left YD}
    Let $\cL$ be a left bialgebroid over $B$, a left-left Yetter-Drinfeld module of $\cL$ is a left $\cL$-comodule  and a left $\cL$-module $\Lambda$, such that
    \begin{itemize}
        \item $s_{L}(b)\bla \rho=b.\rho$ and $t_{L}(b)\bla\rho=\rho.b$, $\forall b\in B, \rho\in \Lambda$.
        \item $(X\o \bla\rho)\mo X\t\ot_{B}(X\o \bla\rho)\z=X\o\rho\mo\ot_{B}X\t\bla \rho\z$, $\forall X\in \cL, \rho\in \Lambda$.
    \end{itemize}
    We denoted the category of left-left Yetter-Drinfeld modules of $\cL$ by ${}_{\cL}^{\cL}\mathcal{YD}$.
\end{defi}

\begin{defi}\label{def. left covariant}
    Let $\cL$ be a left bialgebroid  over $B$, a left covariant bimodule of $\cL$ is a left $\cL$-comodule and a $\cL$-bimodule, such that
    \begin{itemize}
        \item [(1)] $s_{L}(b)\la\rho\ra s_{L}(b')=b.\rho.b'$,  $\forall b, b'\in B, \rho\in \Gamma$.
        \item [(2)] $(X\la \rho\ra Y)\mo\ot_{B}(X\la \rho\ra Y)\z=X\o\rho\mo Y\o\ot_{B}X\t\la\rho\z\ra Y\t$, $\forall X, Y\in \cL, \rho\in \Gamma$.
    \end{itemize}
\end{defi}
It is not hard to check (2) is well defined by the conditions in (1).

\begin{defi}
    Let $\cL$ be a left bialgebroid over $B$, a right $\cL$-comodule is a $B$-bimodule $\Gamma$ (with $B$-bimodule structure denoted by $b^{*}\rho{}^{*}b'$, for all $b, b'\in B$), together with a $B$-bimodule map $\delta_{L}: \Gamma\to \Gamma\times_{B}\cL\subseteq\Gamma\ot_{B}\cL$, written $\delta_{L}(\rho)=\rho\z\ot_{B}\rho\o$, such that
    \[(\delta_{L}\ot_{B}\id_{\cL})\circ \delta_{L}=(\id_{\Gamma}\ot_{B}\Delta_{L})\circ \delta_{L},\qquad (\id_{\Gamma}\ot_{B}\varepsilon_{L})\circ \delta_{L}=\id_{\Gamma},\]

where
\begin{equation*} \Gamma\times_{B}\cL :=\{\ \sum_i \rho_i \ot_{B} X_i\in \Gamma\otimes_{B}\cL\ |\ \sum_i {b}^{*}\rho_i  \ot_{B} X_i=
\rho_i  \ot_{B} X_i s_{L}(b),\quad \forall b\in B\ \}.
\end{equation*}
\end{defi}
Here $\delta_{L}$ is a $B$-bimodule map in the sense that
$\delta_{L}(b^{*}\rho^{*}b')=\rho\z \ot_{B} t_{L}(b')\rho\o t_{L}(b)$.

\begin{defi}\label{def. right-left YD}
    Let $\cL$ be a left bialgebroid over $B$, a right-left Yetter-Drinfeld module of $\cL$ is a right $\cL$-comodule  and a left $\cL$-module $\Lambda$, such that
    \begin{itemize}
        \item $s_{L}(b)\bla \rho=b^{*}\rho$ and $t_{L}(b)\bla \rho=\rho^{*}b$, $\forall b\in B, \rho\in \Lambda$.
        \item $(X\t\bla \rho)\z\ot_{B}(X\t\bla \rho)\o X\o=X\o\bla \rho\z\ot_{B}X\t \rho\o$.
    \end{itemize}
    The category of right-left Yetter-Drinfeld modules of $\cL$ is denoted by ${}_{\cL}\mathcal{YD}^{\cL}$.
\end{defi}
A right covariant bimodule of $\cL$ can be similarly defined, which is a right $\cL$ comodule and $\cL$-bimodule with compatible relation.

Similarly, there are also right bialgebroids and their comodules:
\begin{defi}\label{def:right.bgd} Let $A$ be a unital algebra. A right bialgebroid over $A$ (or right $A$-bialgebroid) is an algebra $\cR$ with commute algebra maps (`source' and `target' maps) $s_{R}:A\to \cR$ and $t_{R}:A^{op}\to \cR$ (and hence a $A^e$-ring) and a $A$-coring for the bimodule structure (\ref{eq:rbgd.bimod1}) which is compatible in the sense
\begin{itemize}
\item[(i)] The coproduct $\Delta_{R}$ corestricts to an algebra map  $\cR\to \cR\times_{A} \cR$ where
\begin{equation*} \cR\times_{A} \cR :=\{\ \sum_i X_i \ot_{A} Y_i\ |\ \sum_i s_{R}(a)X_i \ot_{A} Y_i=
\sum_i X_i \ot_{A}  t_{R}(a)Y_i  ,\quad \forall a\in A\ \}\subseteq \cR\tens_{A}\cR,
\end{equation*}
 is an algebra via factorwise multiplication.
\item[(ii)] The counit $\varepsilon$ is a right character in the following sense:
\begin{equation*}\varepsilon(1_{\cR})=1_{A},\quad \varepsilon(s_{R}(\varepsilon(X))Y)=\varepsilon(XY)=\varepsilon(t_{R}(\varepsilon(X))Y)\end{equation*}
for all $X,Y\in \cR$ and $a\in A$.
\end{itemize}
\end{defi}

Morphisms between right $A$-bialgebroids are $A$-coring maps which are also algebra maps. In the following, we will use the sumless Sweedler indexes to denote the coproducts of right bialgebroids, namely,
 $\Delta_{R}(X)=X\ro\ot_{A}X\rt$. Whenever we mention a right bialgebroid (in case we haven't mentioned its base algebra), by default, it is a right bialgebroid over $A$.

\begin{defi}
    Let $\cR$ be a right bialgebroid over $A$, a right $\cR$-comodule is a $A$-bimodule $\Gamma$ (with $A$-bimodule structure denoted by $a^{.}\rho{}^{.}a'$, for all $a, a'\in A$), together with a $A$-bimodule map $\delta_{R}: \Gamma\to \Gamma\times_{A}\cR\subseteq\Gamma\ot_{A}\cR$, written $\delta_{R}(\rho)=\rho\rz\ot_{A}\rho\ro$, such that
    \[(\delta_{R}\ot_{A}\id_{\cR})\circ \delta_{R}=(\id_{\Gamma}\ot_{A}\Delta_{R})\circ \delta_{R},\qquad (\id_{\Gamma}\ot_{A}\varepsilon_{R})\circ \delta_{R}=\id_{\Gamma},\]
where
\begin{equation*} \Gamma\times_{A}\cR :=\{\ \sum_i \rho_i \ot_{A} X_i\in \Gamma\otimes_{A}\cR\ |\ \sum_i a^{.}\rho_i  \ot_{A} X_i=
\rho_i  \ot_{A} t_{R}(a)X_i,\quad \forall a\in A\ \}.
\end{equation*}
\end{defi}

Here $\delta_{R}$ is a $A$-bimodule map in the sense that
$\delta_{R}(a^{.}\rho^{.}a')=\rho\rz \ot_{A} s_{R}(a)\rho\ro s_{R}(a')$.

\begin{defi}\label{def. right-right YD}
    Let $\cR$ be a right bialgebroid over $A$, a right-right Yetter-Drinfeld module of $\cR$ is a right $\cR$-comodule  and a right $\cR$-module $\Lambda$, such that
    \begin{itemize}
        \item $\rho\bra s_{R}(a)=\rho^{.}a$ and $\rho\bra t_{R}(a)=a^{.}\rho$, $\forall a\in A, \rho\in \Lambda$.
        \item $(\rho\bra X\rt)\rz\ot_{A}X\ro(\rho\bra X\rt)\ro=\rho\rz \bra X\ro \ot_{A}\rho\ro X\rt$, $\forall X\in \cR, \rho\in \Lambda$.
    \end{itemize}
    We denoted the category of right-right Yetter-Drinfeld modules of $\cR$ by $\mathcal{YD}_{\cR}^{\cR}$.
\end{defi}

\begin{defi}\label{def. right covariant}
    Let $\cR$ be a right bialgebroid over $A$, a right covariant bimodule of $\cR$ is a right $\cR$-comodule and a $\cR$-bimodule, such that
    \begin{itemize}
        \item [(1)] $s_{R}(a)\la\rho\ra s_{R}(a')=a^{.}\rho^{.}a'$,  $\forall a, a'\in A, \rho\in \Gamma$.
        \item [(2)] $(X\la \rho\ra Y)\rz\ot_{A}(X\la \rho\ra Y)\ro=X\ro\la\rho\rz\ra Y\ro\ot_{A}X\rt\rho\ro Y\rt$, $\forall X, Y\in \cR, \rho\in \Gamma$.
    \end{itemize}
\end{defi}

\begin{defi}
    Let $\cR$ be a right bialgebroid over $A$, a left $\cR$-comodule is a $A$-bimodule $\Gamma$ (with $A$-bimodule structure denoted by $a_{*}\rho_{*}a'$, for all $a, a'\in A$), together with a $A$-bimodule map ${}_{R}\delta: \Gamma\to \cR\times_{A}\Gamma\subseteq\cR\ot_{A}\Gamma$, written ${}_{R}\delta(\rho)=\rho\rmo\ot_{A}\rho\rz$, such that
    \[(\Delta_{R}\ot_{A}\id_{\Gamma})\circ {}_{R}\delta=(\id_{\cR}\ot_{A}{}_{R}\delta)\circ {}_{R}\delta,\qquad (\varepsilon_{R}\ot_{A}\id_{\Gamma})\circ {}_{R}\delta=\id_{\Gamma},\]
where
\begin{equation*} \cR\times_{A} \Gamma :=\{\ \sum_i X_i \ot_{A} \rho_i\in \cR\tens_A\Gamma\ |\ \sum_i s_{R}(a)X_i \ot_{A} \rho_i=
\sum_i X_i \ot_{A}  \rho_i{}_{*}a  ,\quad \forall a\in A\ \}.
\end{equation*}
\end{defi}

Here ${}_{R}\delta$ is a $A$-bimodule map in the sense that
${}_{R}\delta(a_{*}\rho_{*}a')=t_{R}(a')\rho\rmo t_{R}(a)\ot_{A}\rho\rz$.

\begin{defi}\label{def. left-right YD}
    Let $\cR$ be a right bialgebroid over $A$, a left-right Yetter-Drinfeld module of $\cR$ is a left $\cR$-comodule  and a right $\cR$-module $\Lambda$, such that
    \begin{itemize}
        \item $\rho\bra s_{R}(a)=\rho_{*}a$ and $\rho\bra t_{R}(a)=a_{*}\rho$, $\forall a\in A, \rho\in \Lambda$.
        \item $X\rt(\rho\bra X\ro)\rmo\ot_{A}(\rho\bra X\ro)\rz=\rho\rmo X\ro\ot_{A}\rho\rz\bra X\rt$.
    \end{itemize}
    The category of left-right Yetter-Drinfeld modules of $\cR$ is denoted by ${}^{\cR}\mathcal{YD}_{\cR}$.
\end{defi}

A left covariant bimodule of $\cR$ can be similarly defined, which is a left $\cR$-comodule and $\cR$-bimodule with compatible relation.

\subsection{Hopf algebroids}

\begin{defi}\label{defHopf}
A left $B$-bialgebroid $\cL$ is a left Hopf algebroid (\cite{schau1}, Thm and Def 3.5.) if
\[\lambda: \cL\ot_{B^{op}}\cL\to \cL\ot_{B}\cL,\quad
    \lambda(X\ot_{B^{op}} Y)=\one{X}\ot_{B}\two{X}Y\]
is invertible, where $\tens_{B^{op}}$ is induced by $t_{L}$ (so $Xt_{L}(a)\ot_{B^{op}}Y=X\ot_{B^{op}}t_{L}(a)Y$, for all $X, Y\in \cL$ and $b\in B$) while the $\tens_{B}$ is the standard one using (\ref{eq:rbgd.bimod}) (so $t_{L}(a)X\ot_{B}Y=X\ot_{B}s_{L}(a)Y$).\\
Similarly, a left $B$-bialgebroid $\cL$ is an anti-left Hopf algebroid if
\[\mu: \cL\ot^{B^{op}}\cL\to \cL\ot_{B}\cL,\quad
    \mu(X\ot^{B^{op}} Y)=\one{Y}X\ot_{B}\two{Y}\]
is invertible, where $\tens^{B^{op}}$ is induced by $s_{L}$ (so $s_{L}(a)X\ot^{B^{op}}Y=X\ot^{B^{op}}Ys_{L}(a)$, for all $X, Y\in \cL$ and $a\in B$) while the $\tens_{B}$ is the standard one using (\ref{eq:rbgd.bimod}).
\end{defi}

In the following we will always use the balanced tensor product explained as above. If $B=k$ then $\lambda$ reduces to the map $\cL\tens\cL\to \cL\tens\cL$ given by $h\tens g\mapsto h\o\tens h\t g$ which for a usual Hopf algebra has inverse $h\tens g\mapsto h\o\tens S(h\t)g$. We refer to texts such as
\cite{KS,Ma:book,Swe} for more details. We adopt the shorthand
\begin{equation}\label{X+-} X_{+}\ot_{B^{op}}X_{-}:=\lambda^{-1}(X\ot_{B}1),\end{equation}
\begin{equation}\label{X[+][-]} X_{[-]}\ot^{B^{op}}X_{[+]}:=\mu^{-1}(1\ot_{B}X).\end{equation}

Similarly,

\begin{defi}\label{defHopf1}
A right $A$-bialgebroid  $\cR$  is a right Hopf algebroid if
\[\hat{\lambda}: \cR\ot_{A^{op}}\cR\to \cR\ot_{A}\cR,\quad
    \hat{\lambda}(X\ot_{A^{op}} Y)=XY\ro\ot_{A}Y\rt\]
is invertible, where $\tens_{A^{op}}$ is induced by $t_{R}$ (so $Xt_{R}(a)\ot_{A^{op}}Y=X\ot_{A^{op}}t_{R}(a)Y$, for all $X, Y\in \cR$ and $a\in A$) while the $\tens_{A}$ is the standard one (\ref{eq:rbgd.bimod1}) (so $Xs_{R}(a)\ot_{A}Y=X\ot_{A}Yt_{R}(a)$).\\
Similarly, a right $A$-bialgebroid $\cR$ is an anti-right Hopf algebroid if
\[\hat{\mu}: \cR\ot^{A^{op}}\cR\to \cR\ot_{A}\cR,\quad
    \hat{\mu}(X\ot^{A^{op}} Y)=X\ro\ot_{A}YX\rt\]
is invertible, where $\tens^{A^{op}}$ is induced by $s_{R}$ (so $s_{R}(a)X\ot^{A^{op}}Y=X\ot^{A^{op}}Ys_{R}(a)$, for all $X, Y\in \cR$ and $a\in A$) while the $\tens_{A}$ is the standard one (\ref{eq:rbgd.bimod1}).
\end{defi}

 We adopt the shorthand
\begin{equation}\label{X+-1} X\m\ot_{A^{op}} X\p :=\hat{\lambda}^{-1}(1\ot_{A}X),\end{equation}
\begin{equation}\label{X[+][-]1} X\np\ot^{A^{op}}X\nm:=\hat{\mu}^{-1}(X\ot_{A}1).\end{equation}

In the following, we are going to introduce a more symmetric Hopf algebroid, namely, a bialgebroid with bijective antipode that is given in \cite{BS}

\begin{defi}\label{def. full Hopf algebroid1}
    A left $B$-bialgebroid $\cL$ is a Hopf algebroid, if there is an invertible anti-algebra map $S:\cL\to \cL$, such that
    \begin{itemize}
        \item $S\circ t_{L}=s_{L}$,
        \item $\one {(S^{-1}\two{X})} \ot_B \two {(S^{-1}\two{X})}\one{X}  = S^{-1}(X) \ot_B 1_\cL$
        \item $\one {S(\one{X})} \two{X} \ot_{B} \two{S(\one{X})}  = 1_\cL \ot_{B} S(X)$.
    \end{itemize}
\end{defi}
We can also define a Hopf algebroid in terms of a right bialgebroid.
\begin{defi}\label{def. full Hopf algebroid3}
    A right $A$-bialgebroid $\cR$  is a Hopf algebroid, if there is an invertible anti-algebra map $S:\cR\to \cR$, such that
    \begin{itemize}
        \item $S\circ t_{R}=s_{R}$,
        \item $X\rt S^{-1}(X\ro)\ro\ot_{A}S^{-1}(X\ro)\rt=1\ot_{A}S^{-1}(X)$,
        \item $S(X\rt)\ro\ot_{A}X\ro S(X\rt)\rt=S(X)\ot_{A}1$.
    \end{itemize}
\end{defi}

There is another equivalent definition given by \cite{BS}

\begin{defi}\label{def. full Hopf algebroid2}
    A Hopf algebroid consists of a left bialgebroid $(\cL, \varepsilon_{L}, \Delta_{L}, s_{L}, t_{L}, m)$ over $B$ and a right bialgebroid $(\cR, \varepsilon_{R}, \Delta_{R}, s_{R}, t_{R}, m)$ over $A$, such that
    \begin{itemize}
        \item $A$ is isomorphic to $B^{op}$,  $\cL$ and $\cR$ have the same underlying algebra structure $H$;
        \item $s_{L}(B)=t_{R}(A)$ and $t_{L}(B)=s_{R}(A)$ as subrings of $H$;
        \item $(\id\ot_{B}\Delta_{R})\circ\Delta_{L}=(\Delta_{L}\ot_{A}\id)\circ\Delta_{R}$ and $(\id\ot_{A}\Delta_{L})\circ\Delta_{R}=(\Delta_{R}\ot_{B}\id)\circ\Delta_{L}$;
        \item $\cL$ is a left Hopf and anti-left Hopf algebroid. (Or equivalently, $\cR$ is a right Hopf and anti-right Hopf algebroid.)
    \end{itemize}

\end{defi}

Given a Hopf algebroid  $(\cL, \varepsilon_{L}, \Delta_{L}, s_{L}, t_{L}, m)$ with $S$ as in definition \ref{def. full Hopf algebroid1}, we can construct a right Hopf algebroid $(\cR, \varepsilon_{R}, \Delta_{R}, s_{R}, t_{R}, m)$ over $A$ with $A:=B^{op}$ and $\cR:=\cL$ as algebra; the source and target maps are given by

\begin{align}\label{equ. left source target maps to right source and target maps}
  s_{R}(a):=t_{L}(a), \quad t_{R}(a):=S^{-1}\circ t_{L}(a);
\end{align}
the coproduct is given by
\begin{align}\label{equ. left coproduct to right coproduct}
    \Delta_{R}(X):=S(S^{-1}(X)\t)\ot_{A}S(S^{-1}(X)\o)=S^{-1}(S(X)\t)\ot_{A}S^{-1}(S(X)\o);
\end{align}
the counit is given by
\begin{align}\label{equ. left counit to right counit}
\varepsilon_{R}:=\varepsilon_{L}\circ S.
\end{align}
We can also see $\cL$ is a left Hopf algebroid and an anti-left Hopf algebroid, while $\cR$ is a right Hopf algebroid and anti-right Hopf algebroid, with
\begin{align}   X_{+}\ot_{B^{op}}X_{-}=X\ro\ot_{B^{op}}S(X\rt),\quad X_{[-]}\ot^{B^{op}}X_{[+]}=S^{-1}(X\ro)\ot^{B^{op}}X\rt.
\end{align}
And
\begin{align}
    X\m\ot_{A^{op}}X\p=S(X\o)\ot_{A^{op}}X\t,\quad X\np\ot^{A^{op}}X\nm=X\o\ot^{A^{op}}S^{-1}(X\t).
\end{align}
We also have
\begin{align}
    S^{\pm}(X)\ro\ot_{A}S^{\pm}(X)\rt&=S^{\pm}(X\t)\ot_{A}S^{\pm}(X\o),\\
    S^{\pm}(X)\o\ot_{B}S^{\pm}(X)\t&=S^{\pm}(X\rt)\ot_{B}S^{\pm}(X\ro).
\end{align}

 We will always use $\cH$ or $(\cL, \cR, S)$ to denote the Hopf algebroid with the left $B$-bialgebroid $\cL$ and right $A$-bialgebroid $\cR$ with $H$ being the underlying algebra of $\cL$ and $\cR$ as in Definition \ref{def. full Hopf algebroid2}, equipped with an invertible antipode $S$ in the sense of \cite{BS}. Moreover, in order to simplify the relation in $(\cL, \cR, S)$, with no loss of generality, we will always assume the base algebra of $\cL$ is the opposite base algebra of $\cR$, i.e. $A=B^{op}$, with their structure satisfies (\ref{equ. left source target maps to right source and target maps}) and (\ref{equ. left coproduct to right coproduct}).



Given a Hopf algebroid $(\cL, \cR, S)$, a comodule of $(\cL, \cR, S)$ is introduced in \cite{Boehm,BB}
\begin{defi}
    A right comodule of $(\cL, \cR, S)$ is a right $\cL$ and $\cR$ comodule $\Gamma$ (with their corresponding bimodule structure denoted by $b^{*}\rho^{*}b'$ and $a^{.}\rho^{.}a'$ for $a, a'\in A$, $b, b'\in B$), such that
    $\delta_{L}$ is right $A$-module map and $\delta_{R}$ is a right $B$-module map, namely,
         \[\delta_{L}(\rho^{.}a)=\rho\z\ot_{B}\rho\o s_{R}(a), \quad \delta_{R}(\rho^{*}b)=\rho\rz\ot_{A}t_{L}(b)\rho\ro ,\]
      for any $a\in A, b\in B$ and $\rho\in \Gamma$,     and
     \[(\id\ot_{B}\Delta_{R})\circ\delta_{L}=(\delta_{L}\ot_{A}\id)\circ\delta_{R},\quad (\id\ot_{A}\Delta_{L})\circ\delta_{R}=(\delta_{R}\ot_{B}\id)\circ\delta_{L}.\]
     As a comodule of $\cL$ and $\cR$, the $A$-bimodule structure on $\Gamma\ot_{B}\cL$ is given by $a^{.}(\rho\ot_{B} X)^{.}a'=\rho\ot_{B} s_{R}(a)Xs_{R}(a')$, and the $B$-bimodule structure on $\Gamma\ot_{A}\cR$ is given by $b^{*}(\rho\ot_{A} X)^{*}b'=\rho\ot_{A} t_{R}(b')Xt_{R}(b)$.
\end{defi}

Similarly,

\begin{defi}
    A left comodule of $(\cL, \cR, S)$ is a left $\cL$ and $\cR$ comodule $\Gamma$ (with their corresponding bimodule structure denoted by $b_{.}\rho_{.}b'$ and $a_{*}\rho_{*}a'$ for $a, a'\in A$, $b, b'\in B$), such that ${}_{L}\delta$ is a left $\cR$-comodule and left $A$-module map and ${}_{R}\delta$ is a left $\cL$-comodule and left $B$-module  map, where the $A$-bimodule structure on $\cL\ot_{B}\Gamma$ is given by $a_{*}(X\ot_{B} \rho)_{*}a'=t_{R}(a')Xt_{R}(a)\ot_{B}\rho$, and the $B$-bimodule structure on $\cR\ot_{A}\Gamma$ is given by $b_{.}(X\ot_{B} \rho)_{.}b'=s_{L}(b)Xs_{L}(b')\ot_{B}\rho$.

\end{defi}

\begin{prop}
    Let $\Gamma$ be a right comodule of $(\cL, \cR, S)$, then $a^{.}\rho^{.} a'=a^{*}\rho^{*}a'$ and $\Gamma^{co\cR}=\Gamma^{co\cL}$,
    where $\Gamma^{co\cL}:=\{\ \eta\in\Gamma \ |\ \delta_{L}(\eta)=\eta\ot_{B}1 \}$ and $\Gamma^{co\cR}:=\{\ \eta\in\Gamma \ |\ \delta_{R}(\eta)=\eta\ot_{A}1 \}$.
      Similarly, a left $(\cL, \cR, S)$-comodule also results in the same left invariant subspace associated with the left $\cL$ and $\cR$ comodule structure. And $a_{.}\rho_{.} a'=b'_{*}\rho_{*}b$ for $s_{L}(b)=t_{R}(a)$ and $s_{L}(b')=t_{R}(a')$.
\end{prop}

\begin{proof}

As  $\delta_{R}(a^{.}\rho)=\delta_{R}(\rho^{*}a)$ and $\delta_{L}(\rho^{.}a)=\delta_{L}(a^{*}\rho)$, we have $a'^{.}\rho^{.} a=a^{*}\rho^{*}a'$.

    If $\rho\in \Gamma^{co\cR}$, then we have
    \begin{align*}
\rho\z\ot_{B}\rho\o=&\rho\z\rz{}^{.}(\varepsilon_{R}(\rho\z\ro))\ot_{B}\rho\o
=\rho\rz{}^{.}(\varepsilon_{R}(\rho\ro\o))\ot_{B}\rho\ro\t=\rho\ot_{B}1.
    \end{align*}
    So $\Gamma^{co\cR}\subseteq\Gamma^{co\cL}$. Similarly, $\Gamma^{co\cL}\subseteq\Gamma^{co\cR}$, so $\Gamma^{co\cL}=\Gamma^{co\cR}$. The result of a left comodule is slightly different as $s_{L}$ is in general not equal to $t_{R}$.
\end{proof}

\begin{defi}
    A full right-right Yetter-Drinfeld module of $(\cL, \cR, S)$ is a right-right Yetter-Drinfeld module of $\cR$ as well as a right comodule of $(\cL, \cR, S)$, such that the underlying right $\cR$-comodule structure of the right-right Yetter-Drinfeld module is the same as the underlying right $\cR$-comodule structure of right comodule of $(\cL, \cR, S)$. We denoted the category of full right-right Yetter-Drinfeld modules of $\cR$ by $\mathcal{YD}_{\cR}^{\cH}$. \\
    A full left-right Yetter-Drinfeld module of $(\cL, \cR, S)$ is a left-right Yetter-Drinfeld module of $\cR$ as well as a left comodule of $(\cL, \cR, S)$, such that the underlying left $\cR$-comodule structure of the left-right Yetter-Drinfeld module is the same as the underlying left $\cR$-comodule structure of the left comodule of $(\cL, \cR, S)$. We denoted the category of full left-right Yetter-Drinfeld modules of $\cR$ by ${}^{\cH}\mathcal{YD}_{\cR}$. \\A full left-left Yetter-Drinfeld module and full right-left Yetter-Drinfeld module of $\cL$ can be similarly defined, and they are denoted by ${}^{\cH}_{\cL}\mathcal{YD}$ and ${}_{\cL}\mathcal{YD}^{\cH}$ respectively.
\end{defi}

There is a useful Proposition for full Yetter-Drinfeld modules:

\begin{prop}
    Let $(\cL, \cR, S)$ be a Hopf algebroid and $\Lambda\in \mathcal{YD}^{\cR}_{\cR}$, then the Yetter-Drinfeld condition is equivalent to
    \[(\eta\bra X)\rz\ot_{A}(\eta\bra X)\ro=\eta\rz\bra X\t\ro\ot_{A}S(X\o)\eta\ro X\t\rt,\]
    for any $\eta\in \Lambda$.
   \begin{proof}
         We can see
        \begin{align*}
            \eta\rz\bra X\t\ro\ot_{A}S(X\o)\eta\ro X\t\rt=&(\eta\bra X\t\rt)\rz\ot_{A}S(X\o)X\t\ro(\eta\bra X\t\rt)\ro\\
            =&(\eta\bra X\rt)\rz\ot_{A}S(X\ro\o)X\ro\t(\eta\bra X\rt)\ro\\
            =&(\eta\bra X)\rz\ot_{A}(\eta\bra X)\ro.
        \end{align*}
    \end{proof}
\end{prop}

\section{Hopf bimodules of Hopf algebroids}\label{sec3}

\begin{defi}\label{def. bicovariant}
    Let $(\cL, \cR, S)$ be a Hopf algebroid with $H$ being the underlying algebra of $\cL$ (or $\cR$), a Hopf bimodule of $(\cL, \cR, S)$ is a left $\cL$-covariant comodule, a right $\cR$-covariant comodule $\Gamma$ and a $H$-bimodule, such that
    \begin{itemize}
        \item  $(\id\ot_{B}\delta_{R})\circ{}_{L}\delta=({}_{L}\delta\ot_{A}\id)\circ\delta_{R}$.
    \end{itemize}
\end{defi}

It is not hard to see the relation above is well-defined by the second conditions of Definitions \ref{def. left covariant} and \ref{def. right covariant}. We denote the category of Hopf bimodules of $(\cL, \cR, S)$ by ${}_{H}^{\cL}\mathcal{M}_{H}^{\cR}$. Similarly,

\begin{defi}\label{def. bicovariant1}
    Let $(\cL, \cR, S)$ be a Hopf algebroid with $H$ being the underlying algebra of $\cL$ (or $\cR$), an anti-Hopf bimodule of $(\cL, \cR, S)$ is a left $\cR$-covariant comodule, a right $\cL$-covariant comodule $\Gamma$ and a $H$-bimodule, such that
    \begin{itemize}
        \item  $(\id\ot_{A}\delta_{L})\circ{}_{R}\delta=({}_{R}\delta\ot_{B}\id)\circ\delta_{L}$.
    \end{itemize}
    We denote the category of anti-Hopf bimodules of $(\cL, \cR, S)$ by ${}_{H}^{\cR}\mathcal{M}_{H}^{\cL}$.
\end{defi}

\begin{defi}
    Let $(\cL, \cR, S)$ be a Hopf algebroid with $H$ being the underlying algebra of $\cL$ (or $\cR$), a full Hopf bimodule of $(\cL, \cR, S)$ is a Hopf bimodule of $(\cL, \cR, S)$, as well as a left and right $(\cL, \cR, S)$-comodule, such that
    \begin{itemize}
        \item The underlying right $\cR$ comodule structure of the Hopf bimodule is the same as the underlying right $\cR$ comodule structure of the right $(\cL, \cR, S)$ comodule.
        \item  The underlying left $\cL$ comodule structure of the Hopf bimodule is the same as the underlying left $\cL$ comodule structure of the left $(\cL, \cR, S)$ comodule.
    \end{itemize}

    We denoted the category of full Hopf bimodules of $(\cL, \cR, S)$ by ${}^{\cH}_{H}\mathcal{M}{}_{H}^{\cH}$.
\end{defi}

\begin{prop}\label{prop. prop of full Hopf bimodule}
    Let $\Gamma$ be a full Hopf bimodule of $(\cL, \cR, S)$, then all the left and right coactions are pairwise cocommutative:
    \begin{align*}
       (\id\ot_{B}\delta_{L})\circ{}_{L}\delta=&({}_{L}\delta\ot_{B}\id)\circ\delta_{L},\quad(\id\ot_{A}\delta_{R})\circ{}_{R}\delta=({}_{R}\delta\ot_{A}\id)\circ\delta_{R},\\
       (\id\ot_{B}\delta_{R})\circ{}_{L}\delta=&({}_{L}\delta\ot_{A}\id)\circ\delta_{R},\quad(\id\ot_{A}\delta_{L})\circ{}_{R}\delta=({}_{R}\delta\ot_{B}\id)\circ\delta_{L}.
    \end{align*}
    Moreover, all the left and right actions and left and right coactions are pairwise covariant. In other words, there are two extra left and right covariant structures given by  $({}_{R}\delta, \la, \ra)$ and $(\delta_{L}, \la, \ra)$, where the left covariant structure  is a left $\cR$-comodule and a $\cR$-bimodule, such that
    \begin{itemize}
        \item [(1)] $t_{R}(a)\la\rho\ra t_{R}(a')=a'{}_{*}\rho_{*}a$,  $\forall a, a'\in A, \rho\in \Gamma$.
        \item [(2)] $(X\la \rho\ra Y)\rmo\ot_{B}(X\la \rho\ra Y)\rz=X\ro\rho\rmo Y\ro\ot_{B}X\rt\la\rho\rz\ra Y\rt$, $\forall X, Y\in \cR, \rho\in \Gamma$.
    \end{itemize}
    The right covariant structure  is a right $\cL$-comodule and a $\cL$-bimodule, such that
    \begin{itemize}
        \item [(1)] $t_{L}(b)\la\rho\ra t_{L}(b')=b'^{*}\rho^{*}b$,  $\forall b, b'\in B, \rho\in \Gamma$.
        \item [(2)] $(X\la \rho\ra Y)\z\ot_{A}(X\la \rho\ra Y)\o=X\o\la\rho\z\ra Y\o\ot_{A}X\t\rho\o Y\t$, $\forall X, Y\in \cL, \rho\in \Gamma$.
    \end{itemize}
\end{prop}

\begin{proof}
    As we know
    \begin{align}\label{equ. left coaction equation}
        \rho\z\ot_{B}\rho\o=\rho\z\rz {}^{.}(\varepsilon_{R}(\rho\z\ro))\ot_{B}\rho\o=\rho\rz s_{R}(\varepsilon_{R}(\rho\ro\o))\ot_{B}\rho\ro\t,
    \end{align}
    We have
    \begin{align*}
        ({}_{L}\delta\ot_{B}\id)\circ \delta_{L}(\rho)=&\rho\z\mo\ot_{B}\rho\z\z\ot_{B}\rho\o\\
        =&\rho\rz\mo\ot_{B}\rho\rz\z s_{R}(\varepsilon_{R}(\rho\ro\o))\ot_{B}\rho\ro\t\\
        =&\rho\mo\ot_{B}\rho\z\rz s_{R}(\varepsilon_{R}(\rho\z\ro\o))\ot_{B}\rho\z\ro\t\\
        =&\rho\mo\ot_{B}\rho\z\z\rz s_{R}(\varepsilon_{R}(\rho\z\z\ro))\ot_{B}\rho\z\o\\
        =&\rho\mo\ot_{B}\rho\z\z\ot_{B}\rho\o\\
        =&(\id\ot_{B}\delta_{L})\circ{}_{L}\delta(\rho),
    \end{align*}
    where the 3rd step uses the cocommutativity of $\delta_{R}$ and ${}_{L}\delta$, the 4th step uses the property of right comodule structure of $(\cL, \cR, S)$. Similarly, we have
    \[(\id\ot_{A}\delta_{R})\circ{}_{R}\delta=({}_{R}\delta\ot_{A}\id)\circ\delta_{R},\quad(\id\ot_{A}\delta_{L})\circ{}_{R}\delta=({}_{R}\delta\ot_{B}\id)\circ\delta_{L}.\]
    Now, let's show the extra right covariant structure, by using (\ref{equ. left coaction equation}) we have
    \begin{align*}
        (X\la \rho\ra Y)\z\ot_{A}(X\la \rho\ra Y)\o
        =&(X\ro\la\rho\rz\ra Y\ro)s_{R}(\varepsilon_{R}(X\rt\o\rho\ro\o Y\rt\o))\ot_{B}X\rt\t\rho\ro\t Y\rt\t\\
        =&(X\o\ro\la\rho\z\rz\ra Y\o\ro)s_{R}(\varepsilon_{R}(X\o\rt\rho\z\ro Y\o\rt))\ot_{B}X\t\rho\o Y\t\\
        =&(X\o\la\rho\z\ra Y\o)\rz s_{R}(\varepsilon_{R}((X\o\la\rho\z\ra Y\o)\ro))\ot_{B}X\t\rho\o Y\t\\
        =&X\o\la\rho\z\ra Y\o\ot_{B}X\t\rho\o Y\t.
    \end{align*}
    Similarly, we also have the extra left covariant bimodule structure.
\end{proof}

\begin{thm}\label{thm. category equivalence}
    Let $(\cL, \cR, S)$ be a Hopf algebroid, there is a one to one correspondence between the $(\cL, \cR, S)$-Hopf bimodules and the right-right Yetter-Drinfeld modules of $\cR$. Moreover, there is a one to one correspondence between the full Hopf bimodules and the full right-right Yetter-Drinfeld modules of $(\cL, \cR, S)$.
\end{thm}
\begin{proof}
    Let $\Gamma$ be a $(\cL, \cR, S)$-Hopf bimodule, and ${}^{co\cL}\Gamma$ be the left invariant subspace of $\Gamma$. More precisely,
    \[{}^{co\cL}\Gamma=\{\ \eta\in\Gamma \ |\ {}_{L}\delta(\eta)=1\ot_{B}\eta\}.\]
 Now, we are going to show ${}^{co\cL}\Gamma$ is a right-right Yetter-Drinfeld module of $\cR$.
 First, the right action is given by
 \begin{align*}
     \eta\bra X:=X\m\la \eta\ra X\p,
 \end{align*}
for all $X\in \cR$ and $\eta\in \Lambda_{\Gamma}$. This is well defined, since for any $\eta\in {}^{co\cL}\Gamma$, we have ${}_{L}\delta(t_{R}(a)\la \eta)={}_{L}\delta(\eta\ra t_{R}(a))$, therefore, $t_{R}(a)\la \eta=\eta\ra t_{R}(a)$.
We can see $X\m\la \eta\ra X\p$ is left invariant
\begin{align*}
    {}_{L}\delta(X\m\la \eta\ra X\p)=S(X\o)\o X\t\ot_{B}S(X\o)\t \eta X\th=1\ot_{B}X\m\la \eta\ra X\p,
\end{align*}
and it also commutes with $t_{R}(a)$ by (\ref{equ. source and target map with hlambda inv 5}). So the right action is well defined.

The right coaction is defined to be the restriction of $\delta_{R}$ on the subspace  ${}^{co\cL}\Gamma$ with the same underlying $A$-bimodule structure. This is also well defined as $\delta_{R}$ cocommutes with ${}_{L}\delta$, so $\delta_{R}({}^{co\cL}\Gamma)\subseteq {}^{co\cL}\Gamma\times_{A}\cR$. 
We can also see
\begin{align*}
    \eta\bra s_{R}(a)=\eta\ra s_{R}(a)=\eta^{.}a, \quad \eta\bra t_{R}(a)=s_{R}(a)\la \eta=a^{.}\eta.
\end{align*}
Finally,
\begin{align*}
    (\eta\bra X\rt)\rz\ot_{A}X\ro(\eta\bra X\rt)\ro=&X\rt\m\ro\la\eta\rz\ra X\rt\p\ro\ot_{A}X\ro X\rt\m\rt \eta\ro X\rt\p\rt\\
    =&X\rt\p\m\la\eta\rz\ra X\rt\p\p\ro\ot_{A}X\ro X\rt\m \eta\ro X\rt\p\p\rt\\
    =&X\ro\m\la\eta\rz\ra X\ro\p\ot_{A}\eta\ro X\rt\\
    =&\eta\rz \bra X\ro \ot_{A}\eta\ro X\rt.
\end{align*}

Conversely, if $\Lambda$ is a right-right Yetter-Drinfeld module of $\cR$,
define $\Gamma_{\Lambda}:=\cL\ot_{B^{op}}\Lambda=\cL\ot_{A}\Lambda$, where $\tens_{B^{op}}$ is induced by $t_{L}$ (so $Xt_{L}(a)\ot_{B^{op}}\eta=X\ot_{B^{op}}a^{.}\eta$, or equivalently, $Xs_{R}(a)\ot_{A}\eta=X\ot_{A}\eta\bra t_{R}(a)$ by the $A$-bimodule structure on $\cR$ and $\Lambda$ for all $X\in \cR, \eta\in\Lambda$ and $b\in B$). Now, we are going to show $\Gamma_{\Lambda}$ is a Hopf bimodule of $(\cL, \cR, S)$.
First, let's show the left covariant bimodule structure of $\cL$. As a left covariant bimodule of $\cL$, the left coaction is given by
\begin{align*}
    {}_{L}\delta(X\ot_{B^{op}}\eta):=X\o\ot_{B}X\t\ot_{B^{op}}\eta,
\end{align*}
where the underlying $B$-bimodule associate to the left $\cL$ comodule structure is $b.(X\ot\eta).b'=s_{L}(b)Xs_{L}(b)\ot_{B^{op}}\eta$. It is not hard to see this is a well defined comodule structure. The $\cL$-bimodule structure is given by
\begin{align*}
    Y\la(X\ot_{A}\eta)\ra Z:=YXZ\ro\ot_{A}\eta \bra Z\rt.
\end{align*}
We can see this formula factor though $Z\ro\ot_{A}Z\rt$. 
And
\begin{align*}
    s_{L}(a)\la(X\ot_{B^{op}}\eta)\ra s_{L}(a')=s_{L}(a)Xs_{L}(a')\ot_{B^{op}}\eta=a.(X\ot_{B^{op}}\eta).a',
\end{align*}
 since $s_{L}(B)=t_{R}(A)$ as subrings of $\cL$.
 Moreover,
\begin{align*}
    (Y\la (X\ot_{B^{op}}&\eta)\ra Z)\mo\ot_{B}(Y\la (X\ot_{B^{op}}\eta)\ra Z)\z\\
    =&Y\o X\o Z\ro\o\ot_{B}Y\t X\t Z\ro\t\ot_{B^{op}}\eta\bra Z\rt\\
    =&Y\o X\o Z\o\ot_{B}Y\t X\t Z\t\ro\ot_{B^{op}}\eta\bra Z\t\rt,
\end{align*}
where the last step use $(\id\ot_{B}\Delta_{R})\circ\Delta_{L}=(\Delta_{L}\ot_{A}\id)\circ\Delta_{R}$. So we have a left covariant bimodule structure on $\cL\ot_{B^{op}}\Lambda$. Second, let's show the right covariant bimodule structure on $\cL\ot_{B^{op}}\Lambda$, the right coaction is given by
\begin{align*}
    \delta_{R}(X\ot_{B^{op}}\eta):=X\ro\ot_{B^{op}}\eta\rz\ot_{A}X\rt\eta\ro,
\end{align*}
where the underlying $A$-bimodule structure associate to this right $\cR$-comodule is $a^{.}(X\ot_{B^{op}}\eta)^{.}a'=s_{R}(a)X\ot_{B^{op}}\eta^{.}a'$. We can see this is a well defined coaction as it factors though $X\ot_{B^{op}}\eta$.
And
\begin{align*}
    s_{R}(a)\la(X\ot_{B^{op}}\eta)\ra s_{R}(a')=&s_{R}(a)X\ot_{B^{op}}\eta\bra s_{R}(a)=s_{R}(a)X\ot_{B^{op}}\eta^{.}a'\\
    =&a^{.}(X\ot_{B^{op}}\eta)^{.}a'.
\end{align*}
Moreover,
\begin{align*}
    (Y\la (X\ot_{B^{op}}&\eta)\ra Z)\rz\ot_{A}(Y\la (X\ot_{B^{op}}\eta)\ra Z)\ro\\
    =&Y\ro X\ro Z\ro\ot_{B^{op}}(\eta\bra Z\rth)\rz\ot_{A}Y\rt X\rt Z\rt(\eta\bra Z\rth)\ro\\
    =&Y\ro X\ro Z\ro\ot_{B^{op}}\eta\rz\bra Z\rt\ot_{A}Y\rt X\rt \eta\ro Z\rth\\
    =&Y\ro\la (X\ot_{B^{op}}\eta)\rz \ra Z\ro\ot_{A}Y\rt (X\ot_{B^{op}}\eta)\ro Z\rt.
\end{align*}
So $\cL\ot_{B^{op}}\Lambda$ is a right covariant bimodule of $\cR$. Finally,
\begin{align*}
    (\id\ot_{B}\delta_{R})\circ {}_{L}\delta(X\ot_{B^{op}}\eta)=&X\o\ot_{B}X\t\ro\ot_{B^{op}}\eta\rz\ot_{A}X\t\rt\eta\ro\\
    =&X\ro\o\ot_{B}X\ro\t\ot_{B^{op}}\eta\rz\ot_{A}X\rt\eta\ro\\
    =&({}_{L}\delta\ot_{A}\id)\circ\delta_{R}(X\ot_{B^{op}}\eta).
\end{align*}

Let $\Lambda$ be a right-right Yetter-Drinfeld module of $\cR$. It is clear that $\Lambda\simeq {}^{co\cL}(\cL\ot_{B^{op}}\Lambda)$.
Conversely, let $\Gamma$ be a Hopf bimodule of $(\cL, \cR, S)$. Let's show $\Phi: \Gamma\to \cL\ot_{B^{op}}{}^{co\cL}\Gamma$ given by
\begin{align*}
    \Phi(\rho):=\rho\mo{}_{+}\ot_{B^{op}}\rho\mo{}_{-}\la\rho\z
\end{align*}
is an isomorphism. The image of $\Phi$ belongs to $\cL\ot_{B^{op}}{}^{co\cL}\Gamma$, indeed
\begin{align*}    (\id\ot_{B^{op}}{}_{L}\delta)=&\rho\mo{}_{+}\ot_{B^{op}}\rho\mo{}_{-}\o\rho\z \mo\ot_{B}\rho\mo{}_{-}\t\la\rho\z\z\\
    =&\rho\mt{}_{++}\ot_{B^{op}}\rho\mt{}_{-}\rho\mo\ot_{B}\rho\mt{}_{+-}\la\rho\z\\
    =&\rho\mo{}_{+}\ot_{B^{op}}1\ot_{B}\rho\mo{}_{-}\la\rho\z.
\end{align*}
We can also see $\Phi$ is a morphism between Hopf bimodule.

\begin{align*}
    \Phi(Y\la\rho)=&Y\o{}_{+}\rho\mo{}_{+}\ot_{B^{op}}(\rho\mo{}_{-}Y\o{}_{-}Y\t)\la\rho\z=Y\la\Phi(\rho),
\end{align*}
 and
\begin{align*}
    \Phi(\rho\ra Z)=&\rho\mo{}_{+}Z\o{}_{+}\ot_{B^{op}}Z\o{}_{-}\la(\rho\mo{}_{-}\la\rho\z)\ra Z\t\\
    =&\rho\mo{}_{+}Z\ro\ot_{B^{op}}Z\rt\m\la(\rho\mo{}_{-}\la\rho\z)\ra Z\rt\p\\
    =&\rho\mo{}_{+}Z\ro\ot_{B^{op}}(\rho\mo{}_{-}\la\rho\z)\bra Z\rt\\
    =&\Phi(\rho)\ra Z,
\end{align*}
where the second step uses
\begin{align*}
    Z\o{}_{+}\ot_{B^{op}}Z\o{}_{-}\ot_{A^{op}}Z\t=&S^{-1}(S(Z\o)\t)\ot_{B^{op}}S(Z\o)\o\ot\ot_{A^{op}}Z\t\\
    =&Z\o\ro\ot_{B^{op}}S(Z\o\rt)\ot_{A^{op}}Z\t\\
    =&Z\ro\ot_{B^{op}}Z\rt\m\ot_{A^{op}}Z\rt\p.
\end{align*}
So $\Phi$ is $\cL$-bilinear. We can also see $\Phi(b.\rho.b')=b.\Phi(\rho).b'$ and left $\cL$-colinear
\begin{align*}
    {}_{L}\delta\circ\Phi(\rho)=&\rho\mo{}_{+}\o\ot_{B}\rho\mo{}_{+}\t\ot_{B^{op}}\rho\mo{}_{-}\la\rho\z\\
    =&\rho\mt\ot_{B}\rho\mo{}_{+}\ot_{B^{op}}\rho\mo{}_{-}\la\rho\z\\
    =&(\id\ot_{B}\Phi)\circ{}_{L}\delta(\rho).
\end{align*}
Similarly, $\Phi(a^{.}\rho^{.}a')=a^{.}\Phi(\rho)^{.}a'$, indeed, by using $t_{L}(a)=s_{R}(a)$, we have
\begin{align*}
    \Phi(a^{.}\rho^{.}a')=&\rho{}_{+}\ot_{B^{op}}(\rho{}_{-}t_{L}(a))\la \rho\ra t_{L}(a')\\
    =&s_{R}(a)\rho_{+}\ot_{B^{op}}\rho_{-}\la\rho\z \ra s_{R}(a'),
\end{align*}
and
\begin{align*}
    \delta_{R}\circ\Phi(\rho)=&\rho\mo{}_{+}\ro\ot_{B^{op}}\rho\mo{}_{-}\ro\la\rho\z\rz\ot_{A}\rho\mo{}_{+}\rt\rho\mo{}_{-}\rt\rho\z\ro\\
    =&S^{-1}(S(\rho\mo)\th)\ot_{B^{op}}S^{-1}(S(S(\rho\mo)\o)\t)\la\rho\z\rz\ot_{A}S^{-1}(S(S(\rho\mo)\o)\o S(\rho\mo)\t)\rho\z\ro\\
    =&S^{-1}(S(\rho\mo)\t)\ot_{B^{op}}S(\rho\mo)\o\la\rho\z\rz\ot_{A}\rho\z\ro\\
    =&\rho\mo{}_{+}\ot_{B^{op}}\rho\mo{}_{-}\la\rho\z\rz\ot_{A}\rho\z\ro\\
    =&\rho\rz\mo{}_{+}\ot_{B^{op}}\rho\rz\mo{}_{-}\la\rho\rz\z\ot_{A}\rho\ro\\
    =&(\Phi\ot_{A}\id)\circ\delta_{R}(\rho),
\end{align*}
where we use $(\id\ot_{B}\delta_{R})\circ{}_{L}\delta=({}_{L}\delta\ot_{A}\id)\circ\delta_{R}$ in the fifth step. So $\Phi$ is a right $\cR$ comodule map. The inverse $\Phi^{-1}$ is given by
\begin{align*}
    \Phi^{-1}(X\ot_{B^{op}}\eta):=X\la\eta,
\end{align*}
for an $X\ot_{B^{op}}\eta\in \cL\ot_{B^{op}}\Lambda_{\Gamma}$. We can see on the one hand

\begin{align*}
    \Phi^{-1}\circ\Phi(\rho)=(\rho\mo{}_{+}\rho\mo{}_{-})\la\rho\z=s_{L}\circ\varepsilon(\rho\mo)\la \rho\z=\varepsilon(\rho\mo). \rho\z=\rho,
\end{align*}
on the other hand,
\begin{align*}
    \Phi\circ\Phi^{-1}(X\ot_{B^{op}}\eta)=X\o{}_{+}\ot_{B^{op}}(X\o{}_{-}X\t)\la\eta=X\ot_{B^{op}}\eta.
\end{align*}

Now, if $\Gamma$ is a full Hopf bimodule, then we want to show the right-right Yetter-Drinfeld module ${}^{co\cL}\Gamma$ together with the right
$\cL$-comodule structure defined to be the restriction of $\delta_{L}$ on the subspace  ${}^{co\cL}\Gamma$ with the same underlying $B$-bimodule structure is a full Yetter-Drinfeld module. As $\delta_{L}$ cocommutes with ${}_{L}\delta$, we have $\delta_{L}({}^{co\cL}\Gamma)\subseteq {}^{co\cL}\Gamma\times_{B}\cL$. 
Thus ${}^{co\cL}\Gamma$ is a well defined right comodule of $(\cL, \cR, S)$ and therefore a full Yetter-Drinfeld module.

Conversely, if $\Lambda$ is a full Yetter-Drinfeld module of $(\cL, \cR, S)$. By the same construction as above, we first define the extra left $\cR$-comodule and right $\cL$-comodule structures by:
\[{}_{R}\delta(X\ot_{A}\eta):=X\ro\ot_{A} X\rt\ot_{A}\eta, \quad\delta_{L}(X\ot_{A}\eta):=X\o\ot_{A}\eta\z\ot_{B}X\t\eta\o,\]
with the underlying $A$ and $B$-bimodule structure give by
\[a_{*}(X\ot_{A}\eta)_{*}a'=t_{R}(a')Xt_{R}(a)\ot_{A}\eta,\quad b^{*}(X\ot_{A}\eta)^{*}b'=t_{L}(b')X\ot_{A}b^{*}\eta. \]
It is not hard to see both of them are well-defined and ${}_{R}\delta$ is a coaction. Moreover, we can also check $\delta_{L}$ is a well defined right coaction of $\cL$. Indeed,
\begin{align*}
    \delta_{L}(Xt_{L}(a)\ot_{A}\eta)=X\o\ot_{A}\eta\z\ot_{B}X\t s_{R}(a)\eta\o=\delta_{L}(X\ot_{A}a^{.}\eta),
\end{align*}
and we can see it factors though $\eta\z\ot_{B}\eta\o$:
\begin{align*}
    X\o\ot_{A}\eta\z\ot_{A}X\t s_{L}(b)\eta\o=&X\o t_{L}(b)\ot_{A}\eta\z\ot_{A}X\t \eta\o=X\o \ot_{A} b^{.}\eta\z\ot_{A}X\t \eta\o\\
    =&X\o \ot_{A} \eta\z{}^{*}b\ot_{A}X\t \eta\o.
\end{align*}
Recall that $a^{.}(X\ot_{B^{op}}\eta)^{.}a'=s_{R}(a)X\ot_{B^{op}}\eta^{.}a'$, it is not hard to see $a^{.}(X\o \ot_{A} \eta)^{.}a'=a'^{*}(X\o \ot_{A} \eta)^{*}a$. Thus, $\delta_{R}$ is a $B$-bimodule map associated to the right $\cL$-comodule structure and $\delta_{L}$ is a $A$-bimodule map associated to the right $\cR$-comodule structure. The conditions
\[(\id\ot_{B}\Delta_{R})\circ\delta_{L}=(\delta_{L}\ot_{A}\id)\circ\delta_{R},\quad (\id\ot_{A}\Delta_{L})\circ\delta_{R}=(\delta_{R}\ot_{B}\id)\circ\delta_{L}.\]
are easy to check, as $\Lambda$ is a comodule of $(\cL, \cR, S)$. As a result, $\cL\ot_{A}\Lambda$ is a right comodule of $(\cL, \cR, S)$. It is also a left comodule of $(\cL, \cR, S)$, as the left coactions are the coproducts on the first factor.

Finally, we can see $\Phi^{-1}: \cL\ot_{A}{}^{co\cL}\Gamma\to \Gamma$ with
  $\Phi^{-1}( X\ot_{A}\eta)= X\la \eta$ is right $\cL$-colinear, because by Proposition \ref{prop. prop of full Hopf bimodule} $\Gamma$ is right $\cL$ covariant. It is also left $\cR$-colinear, because $\Gamma$ is left-$\cR$ covariant. Also it is clear that $\Lambda\simeq {}^{co\cL}(\cL\ot_{B^{op}}\Lambda)$ as full Yetter-Drinfeld modules.
So we have the one-to-one correspondence between the full Hopf bimodules and the full right-right Yetter-Drinfeld modules of $(\cL, \cR, S)$.

\end{proof}

Similarly,
\begin{thm}\label{thm. right invariant of the main theorem}
    Let $(\cL, \cR, S)$ be a Hopf algebroid, there is a one to one correspondence between the $(\cL, \cR, S)$-Hopf bimodules and the left-left Yetter-Drinfeld modules of $\cL$. Moreover, there is a one to one correspondence between the full Hopf bimodules and the full left-left Yetter-Drinfeld modules of $(\cL, \cR, S)$.
\end{thm}

\begin{proof}
    The proof is indeed the same as the right-handed version above.
Given a Hopf bimodule $\Gamma$, we can define a left-left Yetter-Drinfeld module on $\Gamma^{co\cR}=\{\ \eta\in\Gamma \ |\ \delta_{R}(\eta)=\eta\ot_{A}1\}$
with the left $\cL$-action
\[X\bla\rho:=X_{+}\la\rho\ra X_{-}.\]
And the left $\cL$-coaction is given by the restriction of ${}_{L}\delta$ on the subspace $\Gamma^{co\cR}$,
where the underlying $B$-bimodule structure is $b.\rho.b'=s_{L}(b)\la \rho\ra s_{L}(b')$.

Conversely, given a left-left Yetter-Drinfeld module $\Lambda$, we can construct a Hopf bimodule $\Gamma_{\Lambda}=\Lambda\ot_{A^{op}}\cR$.
The left $\cL$-comodule structure is
\[{}_{L}\delta(\eta\ot_{A^{op}} X)=\eta\mo X\o\ot_{B}\eta\z\ot_{A^{op}} X\t,\]
where the underlying $B$-bimodule structure is $b.(\eta\ot_{A^{op}} X).b'=b.\eta\ot_{A^{op}}X s_{L}(b')$.
The right $\cR$-comodule structure is
\[\delta_{R}(\eta\ot_{A^{op}} X)=\eta\ot_{A^{op}} X\ro\ot_{A}X\rt,\]
    where the underlying $A$-bimodule structure is $a^{.}(\eta\ot_{A^{op}} X)^{.}a'=\eta\ot_{A^{op}}s_{R}(a)Xs_{R}(a')$.

    The $\cL$ bimodule structure is
    \[Y\la(\eta\ot_{A^{op}} X)\ra Z=Y\o\bla\eta\ot_{A^{op}}Y\t XZ.\]
\end{proof}

There is also a useful Corollary:
\begin{cor}\label{cor. left bivariant bimodule}
    Let  $\Gamma$ be a Hopf bimodule   of $(\cL, \cR, S)$, then for any    $\rho\in \Gamma$,  there is $\sum\eta_{i}\ot_{B^{e}} X_{i}\in {}^{co\cL}\Gamma\ot_{B^{e}}\cL$, such that  $\rho=\sum\eta_{i}\ra X_{i}$. Similarly, there is $\sum Y_{i}\ot_{B^{e}}\eta_{i}\in \cL\ot_{B^{e}}\Gamma^{co\cR}$, such that $\rho=\sum Y_{i}\la \eta_{i}$.
\end{cor}
\begin{proof}
  For any $\eta\in {}^{co\cL}\Gamma$ and $X\in \cL$,  we can show
    \begin{align*}
        X\la \eta=&S(S^{-1}(X\t)\o)\la\eta\ra(S^{-1}(X\t)\t X\o)
        =(\eta\bra X\nm)\ra X\np.
    \end{align*}
    By Theorem \ref{thm. category equivalence}, any element $\rho\in \Gamma$ can be given by $X\la\eta$ for some $\eta\in {}^{co\cL}\Gamma$ and $X\in \cL$, so we have the result. Similarly,
    we have
    \[\eta\ra X=X_{[+]}\la(X_{[-]}\bla\eta),\]
    for any $\eta\in {}_{\Gamma}\Lambda$.
\end{proof}

For anti-Hopf bimodules, we have
\begin{thm}\label{thm. category equivalence1}
    Let $(\cL, \cR, S)$ be a Hopf algebroid, there is a one-to-one correspondence between the anti-Hopf bimodules of $(\cL, \cR, S)$ and the left-right Yetter-Drinfeld modules of $\cR$ (resp. right-left Yetter-Drinfeld modules of $\cL$). Moreover, there is a one-to-one correspondence between the full Hopf bimodules of $(\cL, \cR, S)$ and the full left-right Yetter-Drinfeld modules of $(\cL, \cR, S)$ (resp. full right-left Yetter-Drinfeld modules of $(\cL, \cR, S)$).
\end{thm}
\begin{proof}
    Let $\Gamma$ be an anti-Hopf bimodule, we can define a left-right Yetter-Drinfeld module of $\cR$ on
    \[\Gamma^{co\cL}:=\{\ \eta\in\Gamma \ |\ \delta_{L}(\eta)= \eta\ot_{B}1\},\]
    with the right action
    \[\eta\bra X:=X\nm\la \eta \ra X\np.\]
   And the left $\cR$-coaction is given by the restriction of ${}_{R}\delta$ on the subspace ${}_{\Gamma}\Upsilon$.

Conversely, given a left-right Yetter-Drinfeld module $\Upsilon$, we can construct an  anti-Hopf bimodule $\Gamma_{\Upsilon}=\Upsilon\ot_{A}\cR$.
The left $\cR$ comodule structure is
\[{}_{R}\delta(\eta\ot_{A} X)= X\ro\eta\rmo\ot_{A}\eta\rz\ot_{A} X\rt.\]
The right $\cL$ comodule structure is
\[\delta_{L}(\eta\ot_{A} X)=\eta\ot_{A}X\o\ot_{B}X\t.\]
The $\cL$ (or $\cR$) bimodule structure is
    \[Y\la(\eta\ot_{A} X)\ra Z=\eta\bra Z\ro\ot_{A}Y X Z\rt.\]
\end{proof}

\section{(Pre-)braided monoidal category of Yetter-Drinfeld modules and Hopf bimodules}\label{sec4}

\subsection{(Pre-)braided monoidal category of  Hopf bimodules}

Let $\Gamma^{1}, \cdots, \Gamma^{n}$ be Hopf bimodules of a Hopf algebroid $(\cL, \cR, S)$, then $\Gamma^{1}\ot_{H}\cdots\ot_{H}\Gamma^{n}$ is also a Hopf bimodule of $(\cL, \cR, S)$,  where $H$ is the underlying algebra structure of $\cL$ and $\cR$ and $\ot_{H}$ is the balanced tensor product over $H$, namely, $\rho^{1}\ra X\ot_{H}\rho^{2}=\rho^{1}\ot_{H} X\la \rho^{2}$ for any $\rho^{i}\in \Gamma^{i}$. The $\cL$-bimodule structure is given by
\[X\la(\rho^{1}\ot_{H}\cdots\ot_{H}\rho^{n})\ra Y=(X\la\rho^{1})\ot_{H}\cdots\ot_{H}(\rho^{n}\ra Y).\]
The left $\cL$ -comodule structure is given by
\[{}_{L}\delta(\rho^{1}\ot_{H}\cdots\ot_{H}\rho^{n})=\rho^{1}\mo\cdots\rho^{n}\mo\ot_{B}(\rho^{1}\z\ot_{H}\cdots\ot_{H}\rho^{n}\z),\]
with the underlying $B$-bimodule structure
\[b.(\rho^{1}\ot_{H}\cdots\ot_{H}\rho^{n}).b'=(s_{L}(b)\la\rho^{1})\ot_{H}\cdots\ot_{H}(\rho^{n}\ra s_{L}(b')).\]
The right $\cR$-comodule structure is given by
\[\delta_{R}(\rho^{1}\ot_{H}\cdots\ot_{H}\rho^{n})=(\rho^{1}\rz\ot_{H}\cdots\ot_{H}\rho^{n}\rz)\ot_{A}\rho^{1}\ro\cdots\rho^{n}\ro,\]
with the underlying $A$-bimodule structure
\[a^{.}(\rho^{1}\ot_{H}\cdots\ot_{H}\rho^{n})^{.}a'=(s_{R}(a)\la\rho^{1})\ot_{H}\cdots\ot_{H}(\rho^{n}\ra s_{R}(a')).\]

It is not hard to see $\Gamma^{1}\ot_{H}\cdots\ot_{H}\Gamma^{n}$ is a Hopf bimodule of $(\cL, \cR, S)$ with the given structure above. The monoidal category of Hopf bimodule of $(\cL, \cR, S)$ is denoted by $({}_{H}^{\cL}\mathcal{M}_{H}^{\cR}, \ot_{H})$.  Moreover, the category of full Hopf bimodules $({}_{H}^{\cH}\mathcal{M}_{H}^{\cH}, \ot_{H})$ is also a monoidal category with the right diagonal coaction of $\cL$ and left diagonal coaction of $\cR$.

\begin{thm}\label{thm. braided Hopf bimodule category}
    Let  $(\cL, \cR, S)$ be a Hopf algebroid, then $({}_{H}^{\cL}\CM_{H}^{\cR}, \ot_{H})$ is a pre-braided monoidal category, with braiding $\sigma: \Gamma^{1}\ot_{H}\Gamma^{2}\to \Gamma^{2}\ot_{H}\Gamma^{1}$ given by
    \[\sigma(\rho^{1}\ot_{H}\rho^{2})=\rho^{1}\mo{}_{+}\la\rho^{2}\rz\ra \rho^{2}\ro\m\ot_{H}\rho^{1}\mo{}_{-}\la\rho^{1}\z\ra\rho^{2}\ro\p,\]
    for all $\rho^{1}\in \Gamma^{1}, \rho^{2}\in\Gamma^{2}$.
    Moreover, the category of full Hopf bimodules $({}^{\cH}_{H}\mathcal{M}{}_{H}^{\cH}, \ot_{H})$ is a braided monoidal category,  with the  inverse of the braiding  given by
    \[\sigma^{-1}(\rho^{2}\ot_{H}\rho^{1})=S^{-1}(\rho^{2}\o)\m\la\rho^{1}\rz\ra S^{-1}(\rho^{1}\rmo)_{+}\ot_{H}S^{-1}(\rho^{2}\o)\p\la \rho^{2}\z\ra S^{-1}(\rho^{1}\rmo)_{-}.\]
\end{thm}
\begin{proof}
    We first check $\sigma$ factors through $\rho^{1}\mo{}_{+}\ot_{B^{op}}\rho^{1}\mo{}_{-}$:
    \begin{align*}
        (\rho^{1}\mo{}_{+}t_{L}(b))&\la\rho^{2}\rz\ra \rho^{2}\ro\m\ot_{H}\rho^{1}\mo{}_{-}\la\rho^{1}\z\ra\rho^{2}\ro\p\\
        =&(\rho^{1}\mo{}_{+}s_{R}(b))\la\rho^{2}\rz\ra \rho^{2}\ro\m\ot_{H}\rho^{1}\mo{}_{-}\la\rho^{1}\z\ra\rho^{2}\ro\p\\
        =&\rho^{1}\mo{}_{+}\la\rho^{2}\rz\ra (t_{R}(b)\rho^{2}\ro)\m\ot_{H}\rho^{1}\mo{}_{-}\la\rho^{1}\z\ra(t_{R}(b)\rho^{2}\ro)\p\\
        =&\rho^{1}\mo{}_{+}\la\rho^{2}\rz\ra (\rho^{2}\ro\m s_{R}(b)) \ot_{H}\rho^{1}\mo{}_{-}\la\rho^{1}\z\ra\rho^{2}\ro\p\\
        =&\rho^{1}\mo{}_{+}\la\rho^{2}\rz\ra \rho^{2}\ro\m \ot_{H}(t_{L}(b)\rho^{1}\mo{}_{-})\la\rho^{1}\z\ra\rho^{2}\ro\p.
    \end{align*}
    Second, we  check $\sigma$ factors through
    $\rho^{1}\ot_{H}\rho^{2}$:
    \begin{align*}        \sigma(\rho^{1}\ra X\ot_{H}\rho^{2})=&(\rho^{1}\mo{}_{+}X\o{}_{+})\la\rho^{2}\rz\ra \rho^{2}\ro\m\ot_{H}(X\o{}_{-}\rho^{1}\mo{}_{-})\la\rho^{1}\z\ra(X\t\rho^{2}\ro\p)\\
    =&(\rho^{1}\mo{}_{+}X\ro)\la\rho^{2}\rz\ra (\rho^{2}\ro\m X\rt\m)\ot_{H}\rho^{1}\mo{}_{-}\la\rho^{1}\z\ra(X\rt\p\rho^{2}\ro\p)\\
    =&\sigma(\rho^{1}\ot_{H}X\la\rho^{2}),
    \end{align*}
    where the 2nd step uses
    \[Z\o{}_{+}\ot_{B^{op}}Z\o{}_{-}\ot_{A^{op}}Z\t=Z\ro\ot_{B^{op}}Z\rt\m\ot_{A^{op}}Z\rt\p,\]
for any $Z\in \cL$. It is a straightforward computation to check $\sigma$ factors through all the rest balanced tensor products.
    One can easily check $\sigma$ is $\cL$-bilinear. To see it is left $\cL$-colinear, we can first see it is $B$-bilinear, namely, $b.\sigma(\rho^{1}\ot_{H}\rho^{2}).b'=\sigma(b.(\rho^{1}\ot_{H}\rho^{2}).b')$.
    And
    \begin{align*}
        {}_{L}\delta\circ\sigma(\rho^{1}\ot_{H}\rho^{2})=&\rho^{1}\mo{}_{+}\o\rho^{2}\rz\mo\rho^{2}\ro\m\o\rho^{2}\ro\p\o\ot_{B}\rho^{1}\mo{}_{+}\t\la \rho^{2}\rz\z\ra \rho^{2}\ro\m\t\\
        \ot_{H}\rho^{1}&\mo{}_{-}\la\rho^{1}\z\ra\rho^{2}\ro\p\t\\
        =&\rho^{1}\mo{}_{+}\o\rho^{2}\rz\mo\ot_{B}\rho^{1}\mo{}_{+}\t\la \rho^{2}\rz\z\ra \rho^{2}\ro\m
        \ot_{H}\rho^{1}\mo{}_{-}\la\rho^{1}\z\ra\rho^{2}\ro\p\\
        =&\rho^{1}\mo\o\rho^{2}\mo\ot_{B}\rho^{1}\mo\t{}_{+}\la \rho^{2}\z\rz\ra \rho^{2}\z\ro\m
        \ot_{H}\rho^{1}\mo\t{}_{-}\la\rho^{1}\z\ra\rho^{2}\z\ro\p\\
        =&\rho^{1}\mt\rho^{2}\mo\ot_{B}\rho^{1}\mo{}_{+}\la \rho^{2}\z\rz\ra \rho^{2}\z\ro\m
        \ot_{H}\rho^{1}\mo{}_{-}\la\rho^{1}\z\ra\rho^{2}\z\ro\p\\
        =&(\id\ot_{B}\sigma)\circ{}_{L}\delta(\rho^{1}\ot_{H}\rho^{2}),
    \end{align*}
where the 1st step uses $\rho\mo{}_{+}\ot_{B^{op}}\rho\mo{}_{-}\la\rho\z\in \cL\ot_{B^{op}}{}^{co\cL}\Gamma$, and the 2nd step uses
\[X\m\o X\p\o\ot_{B}X\m\t\ot_{A^{op}}X\p\t=1\ot_{B}X\m\ot_{A^{op}}X\p,\]
for any $X\in\cL$.
Similarly, $\sigma$ is also a right $\cR$ comodule map. 
Before we check the braided relation, we first observe that
\[\sigma(\eta^{1}\ot_{H}\eta^{2})=\eta^{2}\rz\ot_{H}\eta^{1}\bra\eta^{2}\ro,\]
for $\eta^{1}\ot_{H}\eta^{2}\in {}^{co\cL}\Gamma^{1}\ot_{H}{}^{co\cL}\Gamma^{2}$. Similarly, for $\eta^{1}\ot_{H}\eta^{2}\in\Gamma^{1co\cR}\ot_{H}\Gamma^{2co\cR}$.
\[\sigma(\eta^{1}\ot_{H}\eta^{2})=\eta^{1}\mo\bla \eta^{2}\ot_{H} \eta^{1}\z.\]

By Theorem \ref{thm. category equivalence}, \ref{thm. right invariant of the main theorem} and Corollary \ref{cor. left bivariant bimodule}, let $\Gamma^{1}, \Gamma^{2}, \Gamma^{3}$ be three Hopf bimodules, we know
 ${}^{co\cL}\Gamma^{1}\ot_{H}\Gamma^{2co\cL}\ot_{H}\Gamma^{3co\cL}$ generates $\Gamma^{1}\ot_{H}\Gamma^{2}\ot_{H}\Gamma^{3}$, so it is sufficient to check the braid relation on $\eta^{1}\ot_{H}\eta^{2}\ot_{H}\eta^{3}\in {}^{co\cL}\Gamma^{1}\ot_{H}\Gamma^{2co\cL}\ot_{H}\Gamma^{3co\cL}$. On the one hand,
 \begin{align*}
     (\sigma\ot_{H}\id)&(\id\ot_{H}\sigma)(\sigma\ot_{H}\id)(\eta^{1}\ot_{H}\eta^{2}\ot_{H}\eta^{3})\\
     =&(\sigma\ot_{H}\id)(\id\ot_{H}\sigma)(\eta^{2}\ot_{H}\eta^{1}\ot_{H}\eta^{3})\\
     =&(\sigma\ot_{H}\id)(\eta^{2}\ot_{H}\eta^{3}\ot_{H}\eta^{1})\\
     =&\eta^{2}\mo\bla\eta^{3}\ot_{H}\eta^{2}\z\ot_{H}\eta^{1}.
 \end{align*}
On the other hand,
\begin{align*}
    (\id\ot_{H}\sigma)&(\sigma\ot_{H}\id)(\id\ot_{H}\sigma)(\eta^{1}\ot_{H}\eta^{2}\ot_{H}\eta^{3})\\
     =&(\id\ot_{H}\sigma)(\sigma\ot_{H}\id)(\eta^{1}\ot_{H}\eta^{2}\mo\bla\eta^{3}\ot_{H}\eta^{2}\z)\\
     =&(\id\ot_{H}\sigma)(\eta^{2}\mo\bla\eta^{3}\ot_{H}\eta^{1}\ot_{H}\eta^{2}\z)\\
     =&\eta^{2}\mo\bla\eta^{3}\ot_{H}\eta^{2}\z\ot_{H}\eta^{1}.
\end{align*}

In the following, we are going to show the category of full Hopf bimodules is a braided monoidal category. We first need to show $\sigma$ is right $\cL$  and left $\cR$-colinear:
\begin{align*}
    \delta_{L}&\circ\sigma(\rho^{1}\ot_{H}\rho^{2})\\
    =&\rho^{1}\mo{}_{+}\o\la\rho^{2}\rz\ra\rho^{2}\ro\m\ot_{H}\rho^{1}\mo{}_{-}\o\la\rho^{1}\z\z\ra\rho^{2}\ro\p\o\ot_{B}\rho^{1}\mo{}_{+}\t\rho^{1}\mo{}_{-}\t\rho^{1}\z\o\rho^{2}\ro\p\t\\
    =&\rho^{1}\mo{}_{+}\la\rho^{2}\rz\ra\rho^{2}\ro\m\ot_{H}\rho^{1}\mo{}_{-}\la\rho^{1}\z\z\ra\rho^{2}\ro\p\o\ot_{B}\rho^{1}\z\o\rho^{2}\ro\p\t\\
    =&\rho^{1}\z\mo{}_{+}\la\rho^{2}\rz\ra\rho^{2}\ro\o\m\ot_{H}\rho^{1}\z\mo{}_{-}\la\rho^{1}\z\z\ra\rho^{2}\ro\o\p\ot_{B}\rho^{1}\o\rho^{2}\ro\t\\
    =&\rho^{1}\z\mo{}_{+}\la\rho^{2}\z\rz\ra\rho^{2}\z\ro\m\ot_{H}\rho^{1}\z\mo{}_{-}\la\rho^{1}\z\z\ra\rho^{2}\z\ro\p\ot_{B}\rho^{1}\o\rho^{2}\o\\
    =&(\sigma\ot_{B}\id)\circ\delta_{L}(\rho^{1}\ot_{H}\rho^{2}),
\end{align*}
where the first step uses
$\rho^{2}\rz\ra\rho^{2}\ro\m\ot_{A^{op}}\rho^{2}\ro\p\in \Gamma^{co\cR}\ot_{A^{op}}\cL$; the second step uses $X_{+}\o\ot_{B^{op}}X_{-}\o\ot_{B}X_{+}\t X_{-}\t=X_{+}\ot_{B^{op}}X_{-}\ot_{B}1$; the 3rd step uses $X\m\ot_{A^{op}}X\p\o\ot_{B}X\p\t=X\o\m\ot_{A^{op}} X\o\p\ot_{B} X\t$. Similarly, we can show $\sigma$ is also $\cR$-colinear. 
So $\sigma$ is a morphism in ${}^{\cH}_{H}\mathcal{M}{}_{H}^{\cH}$. Second, we can see $\sigma^{-1}$ factors through all the balanced tensor products. Here we only check $\sigma^{-1}$ factors through the balanced product $\ot_{H}$. On the one hand,
\begin{align*}
    \sigma^{-1}&(\rho^{2}\ra X\ot_{H}\rho^{1})\\=&(S^{-1}(\rho^{2}\o)\m S^{-1}(X\t)\m)\la\rho^{1}\rz\ra S^{-1}(\rho^{1}\rmo)_{+}\ot_{H}(S^{-1}(X\t)\p S^{-1}(\rho^{2}\o)\p)\la \rho^{2}\z\ra (X\o S^{-1}(\rho^{1}\rmo)_{-}),
\end{align*}
on the other hand,

\begin{align*}
    \sigma^{-1}&(\rho^{2} \ot_{H}X\la\rho^{1})\\
    =&(S^{-1}(\rho^{2}\o)\m X\rt)\la\rho^{1}\rz\ra (S^{-1}(\rho^{1}\rmo)_{+}S^{-1}(X\ro)_{+})\ot_{H}S^{-1}(\rho^{2}\o)\p\la \rho^{2}\z\ra (S^{-1}(X\ro)_{-} S^{-1}(\rho^{1}\rmo)_{-}),
\end{align*}
where we use the fact that
\begin{align*}
   S^{-1}(X\ro)_{+}&\ot_{B^{op}}S^{-1}(X\ro)_{-}\ot^{B^{op}} X\rt\\
   =&S^{-1}(X\ro)\ro\ot_{B^{op}}S(S^{-1}(X\ro)\rt)\ot^{B^{op}}X\rt\\
   =&S^{-1}(X\ro\t)\ot_{B^{op}}X\ro\o\ot^{B^{op}}X\rt\\
   =&S^{-1}(X\t\ro)\ot_{B^{op}}X\o\ot^{B^{op}}X\t\rt\\
   =&S^{-1}(X\t)\p\ot_{B^{op}}X\o\ot^{B^{op}}S^{-1}(X\t)\m.
\end{align*}
It is also not hard to see $\sigma^{-1}$ is $\cL$-bilinear.  
Now, it is sufficient to check $\sigma^{-1}$ is the inverse of $\sigma$ restrict on any element $\rho^{1}\ot_{H}\rho^{2}\in {}^{co\cL}\Gamma^{1}\ot_{H}{}^{co\cL}\Gamma^{2}$. We first observe that
\[\sigma^{-1}(\rho^{2}\ot_{H}\rho^{1})=\rho^{1}\bra S^{-1}(\rho^{2}\o)\ot_{H}\rho^{2}\z.\]
So we have
\begin{align*}
    \sigma\circ\sigma^{-1}(\rho^{2}\ot_{H}\rho^{1})=&\rho^{2}\z\rz\ot_{H}\rho^{1}\bra(S^{-1}(\rho^{2}\o)\rho^{2}\z\ro)\\
    =&\rho^{2}\rz\ot_{H}\rho^{1}\bra(S^{-1}(\rho^{2}\ro\t)\rho^{2}\ro\o)\\
    =&\rho^{2}\ot_{H}\rho^{1},
\end{align*}
we also have
\begin{align*}
    \sigma^{-1}\circ\sigma(\rho^{1}\ot_{H}\rho^{2})=&\rho^{1}\bra(\rho^{2}\ro S^{-1}(\rho^{2}\rz\o))\ot_{H}\rho^{2}\rz\z\\
    =&\rho^{1}\bra(\rho^{2}\o\rt S^{-1}(\rho^{2}\o\ro))\ot_{H}\rho^{2}\z\\
    =&\rho^{1}\ot_{H}\rho^{2}.
\end{align*}

\end{proof}

Similarly, let $\Gamma^{1}, \cdots, \Gamma^{n}$ be anti-Hopf bimodules of a Hopf algebroid $(\cL, \cR, S)$, if we define the balance tensor product $\ot^{H}$ induced by opposite action, i.e. $X\la \rho^{1}\ot^{H} \rho^{2}=\rho^{1}\ot^{H} \rho^{2}\ra X$, then $\Gamma^{1}\ot^{H}\cdots\ot^{H}\Gamma^{n}$ is also an anti-Hopf bimodule of $(\cL, \cR, S)$, where $H$ is the underlying algebra structure of $\cL$ and $\cR$. The $\cL$-bimodule structure is given by
\[X\la(\rho^{1}\ot^{H}\cdots\ot^{H}\rho^{n})\ra Y=(\rho^{1}\ra Y)\ot^{H}\cdots\ot^{H}(X\la\rho^{n}).\]
The left $\cR$-comodule structure is given by
\[{}_{R}\delta(\rho^{1}\ot^{H}\cdots\ot^{H}\rho^{n})=\rho^{n}\rmo\cdots\rho^{1}\rmo\ot_{A}(\rho^{1}\rz\ot^{H}\cdots\ot^{H}\rho^{n}\rz).\]
The right $\cL$-comodule structure is given by
\[\delta_{L}(\rho^{1}\ot^{H}\cdots\ot^{H}\rho^{n})=(\rho^{1}\z\ot^{H}\cdots\ot^{H}\rho^{n}\z)\ot_{B}\rho^{n}\o\cdots\rho^{1}\o.\]
It is not hard to see $\Gamma^{1}\ot^{H}\cdots\ot^{H}\Gamma^{n}$ is an anti-Hopf bimodule of $(\cL, \cR, S)$ with the given structure above. The monoidal category of anti-Hopf bimodule of $(\cL, \cR, S)$ is denoted by $({}_{H}^{\cR}\CM_{H}^{\cL}, \ot^{H})$.  Moreover, the reverse monoidal category of $({}^{\cH}_{H}\mathcal{M}{}_{H}^{\cH}, \ot_{H})$
 denoted by $({}^{\cH}_{H}\mathcal{M}{}_{H}^{\cH}, \ot^{H})$ is also a monoidal category with the opposite left diagonal coaction of $\cL$ and opposite right diagonal coaction of $\cR$. More precisely,
the left $\cL$  comodule structure is given by
\[{}_{L}\delta(\rho^{1}\ot^{H}\cdots\ot^{H}\rho^{n})=\rho^{n}\mo\cdots\rho^{1}\mo\ot_{B}(\rho^{1}\z\ot^{H}\cdots\ot^{H}\rho^{n}\z),\]
the right $\cR$ comodule structure is given by
\[\delta_{R}(\rho^{1}\ot^{H}\cdots\ot^{H}\rho^{n})=(\rho^{1}\rz\ot^{H}\cdots\ot^{H}\rho^{n}\rz)\ot_{A}\rho^{n}\ro\cdots\rho^{1}\ro.\]

\begin{thm}\label{thm. braided Hopf bimodule category1}
    Let  $(\cL, \cR, S)$ be a Hopf algebroid, then $({}_{H}^{\cR}\mathcal{M}_{H}^{\cL}, \ot^{H})$ is a pre-braided monoidal category, with braiding $\tilde{\sigma}: \Gamma^{1}\ot^{H}\Gamma^{2}\to \Gamma^{2}\ot^{H}\Gamma^{1}$ given by
    \[\tilde{\sigma}(\rho^{1}\ot^{H}\rho^{2})=\rho^{2}\o{}_{[-]}\la\rho^{2}\z\ra \rho^{1}\rmo\np\ot^{H}\rho^{2}\o{}_{[+]}\la\rho^{1}\rz\ra\rho^{1}\rmo\nm,\]
    for all $\rho^{1}\in \Gamma^{1}, \rho^{2}\in\Gamma^{2}$.
    Moreover,  $({}^{\cH}_{H}\mathcal{M}{}_{H}^{\cH}, \ot^{H})$ is a braided monoidal category,  with the  inverse of the braiding given by
    \[\tilde{\sigma}^{-1}(\rho^{2}\ot^{H}\rho^{1})=S(\rho^{1}\mo)\np\la\rho^{1}\z\ra S(\rho^{2}\ro)_{[-]}\ot^{H}S(\rho^{1}\mo)\nm\la \rho^{2}\rz\ra S(\rho^{2}\ro)_{[+]}.\]
\end{thm}

\subsection{(Pre-)braided monoidal category of Yetter-Drinfeld modules}
Given a right bialgebroid $\cR$, it is not hard to see that the category $\mathcal{YD}^{\cR}_{\cR}$
 is monoidal with the balanced tensor product $\ot_{A}$ and diagonal (co)action. More precisely, given two Yetter-Drinfeld modules $\Lambda^{1}, \Lambda^{2}\in \mathcal{YD}^{\cR}_{\cR}$, the balanced tensor product is inducted by the $A$-bimodule structure:  $\eta^{1}\ot_{A}\eta^{2}\bra t_{R}(a)=\eta^{1}\bra s_{R}(a)\ot_{A}\eta^{2}$ for any $a\in A$, $\eta^{1}\in \Lambda^{1}$ and $\eta^{2}\in \Lambda^{2}$. The action is given by
 \[(\eta^{1}\ot_{A}\eta^{2})\bra X=\eta^{1}\bra X\ro\ot_{A}\eta^{2}\bra X\rt,\]
and the coaction is given by
\[
\delta_{R}(\eta^{1}\ot_{A}\eta^{2})=\eta^{1}\rz\ot_{A}\eta^{2}\rz\ot_{A}\eta^{1}\ro\eta^{2}\ro,
\]
where the $A$-bimodule structure is given by $a^{.}(\eta^{1}\ot_{A}\eta^{2})^{.}a'=a^{.}\eta^{1}\ot_{A}\eta^{2}{}^{.}a'$. It is not hard to see the tensor product of two Yetter-Drinfeld modules also satisfies the Yetter-Drinfeld condition.
We can also see
 $\mathcal{YD}^{\cH}_{\cR}$ is monoidal category equipped with right $\cL$ diagonal coaction. More precisely,  if $\Lambda^{1}, \Lambda^{2}\in \mathcal{YD}^{\cH}_{\cR}$, and $\eta^{1}\ot_{A}\eta^{2}\in \Lambda^{1}\ot_{A}\Lambda^{2}$,
\[\delta_{L}(\eta^{1}\ot_{A}\eta^{2}):=\eta^{1}\z\ot_{A}\eta^{2}\z\ot_{B}\eta^{1}\o\eta^{2}\o,\]
where the $B$-bimodule structure is given by
$b'^*(\eta^{1}\ot_{A}\eta^{2})^{*}b=\eta^{1*}b\ot_{A}b'^{*}\eta^{2}$.
It is also not hard to see the pair of diagonal coaction $(\delta_{L}, \delta_{R})$ make $\Lambda^{1}\ot_{A} \Lambda^{2}$ a right $(\cL, \cR, S)$ comodule.
Similarly, ${}^{\cR}\mathcal{YD}_{\cR}$ and ${}^{\cH}\mathcal{YD}_{\cR}$ are also monoidal categories with diagonal actions. And its left $\cR$ coaction is given by:
\begin{align*}
    {}_{R}\delta(\eta^{1}\ot_{A}\eta^{2})=\eta^{2}\rmo\eta^{1}\rmo\ot_{A}\eta^{1}\rz\ot_{A}\eta^{2}\rz,
\end{align*}
and for ${}^{\cH}\mathcal{YD}_{\cR}$, the left $\cL$ coaction is
\begin{align*}
    {}_{L}\delta(\eta^{1}\ot_{B}\eta^{2})=\eta^{2}\mo\eta^{1}\mo\ot_{B}\eta^{1}\z\ot_{A}\eta^{2}\z.
\end{align*}

\begin{thm}\label{thm. the category of Yetter-Drinfeld module is braided}
  Let $\cR$ be  a right bialgebroid, then $\mathcal{YD}^{\cR}_{\cR}$ is a pre-braided monoidal category. If $(\cL, \cR, S)$ is a Hopf algebroid, then $\mathcal{YD}^{\cH}_{\cR}$ is a braided monoidal category.
\end{thm}

\begin{proof}
   Given two Yetter-Drinfeld modules $\Lambda^{1}, \Lambda^{2}\in \mathcal{YD}^{\cR}_{\cR}$, and $\eta^{1}\in \Lambda^{1}$ and $\eta^{2}\in \Lambda^{2}$. The braiding is given by

   \begin{align}       \sigma(\eta^{1}\ot_{A}\eta^{2})=\eta^{2}\rz\ot_{A}\eta^{1}\bra\eta^{2}\ro,
   \end{align}
  We can check $\sigma$ factors through all the balanced tensor products and $A$-bilinear. We can also see $\sigma$ is right $\cR$-linear:
\begin{align*}
    \sigma((\eta^{1}\ot_{A}\eta^{2})\bra X)=&(\eta^{2}\bra X\rt)\rz\ot (\eta^{1}\bra X\ro)\bra (\eta^{2}\bra X\rt)\ro\\
    =&\eta^{2}\rz\bra X\ro \ot_{A} \eta^{1}\bra(\eta^{2}\ro X\rt)\\
    =&\sigma(\eta^{1}\ot_{A}\eta^{2})\bra X.
\end{align*}
It is also right $\cR$-colinear:
\begin{align*}
    \delta_{R}(\sigma(\eta^{1}\ot_{A}\eta^{2}))=&\eta^{2}\rz\rz\ot_{A}(\eta^{1}\bra\eta^{2}\ro)\rz\ot_{A}\eta^{2}\rz\ro(\eta^{1}\bra\eta^{2}\ro)\ro\\
    =&\eta^{2}\rz\ot_{A}\eta^{1}\rz\bra\eta^{2}\ro\ot_{A}\eta^{1}\ro\eta^{2}\rt\\
    =&(\sigma\ot_{A}\id)\circ \delta_{R}(\eta^{1}\ot_{A}\eta^{2}).
\end{align*}
Finally, given three Yetter-Drinfeld modules $\Lambda^{1}, \Lambda^{2}, \Lambda^{3}\in \mathcal{YD}^{\cR}_{\cR}$, and $\forall \eta^{1}\ot_{A}\eta^{2}\ot_{A}\eta^{3}\in \Lambda^{1}\ot_{A}\Lambda^{2}\ot_{A}\Lambda^{3}$
\begin{align*}
     (\sigma\ot_{A}\id)&(\id\ot_{A}\sigma)(\sigma\ot_{A}\id)(\eta^{1}\ot_{A}\eta^{2}\ot_{A}\eta^{3})\\
     =&(\sigma\ot_{A}\id)(\id\ot_{A}\sigma)(\eta^{2}\rz\ot_{A}\eta^{1}\bra\eta^{2}\ro\ot_{A}\eta^{3})\\
     =&(\sigma\ot_{A}\id)(\eta^{2}\rz\ot_{A}\eta^{3}\rz\ot_{A}\eta^{1}\bra(\eta^{2}\ro\eta^{3}\ro))\\
     =&\eta^{3}\rz\ot_{A}\eta^{2}\rz\bra\eta^{3}\ro \ot_{A}\eta^{1}\bra(\eta^{2}\ro\eta^{3}\rt).
 \end{align*}
On the other hand,
\begin{align*}
    (\id\ot_{A}\sigma)&(\sigma\ot_{A}\id)(\id\ot_{A}\sigma)(\eta^{1}\ot_{A}\eta^{2}\ot_{A}\eta^{3})\\
     =&(\id\ot_{A}\sigma)(\sigma\ot_{A}\id)(\eta^{1}\ot_{A}\eta^{3}\rz\ot_{A}\eta^{2}\bra\eta^{3}\ro)\\
     =&(\id\ot_{A}\sigma)(\eta^{3}\rz\ot_{A}\eta^{1}\bra\eta^{3}\ro\ot_{A}\eta^{2}\bra\eta^{3}\rt)\\
     =&\eta^{3}\rz\ot_{A}(\eta^{2}\bra\eta^{3}\rt)\rz\ot_{A}\eta^{1}\bra(\eta^{3}\ro( \eta^{2}\bra\eta^{3}\rt)\ro)\\
     =&\eta^{3}\rz\ot_{A}\eta^{2}\rz\bra\eta^{3}\ro \ot_{A}\eta^{1}\bra(\eta^{2}\ro\eta^{3}\rt).
\end{align*}

So  $\mathcal{YD}^{\cR}_{\cR}$ is a pre-braided monoidal category. We can also check $\mathcal{YD}^{\cH}_{\cR}$ is a braided monoidal category. More precisely, the inverse of the braiding is

\[\sigma^{-1}(\eta^{2}\ot_{A}\eta^{1})=\eta^{1}\bra S^{-1}(\eta^{2}\o)\ot_{A}\eta^{2}\z.\]
We will skip the proof that $\sigma^{-1}$ is indeed the inverse of $\sigma$, as this is similar to the inverse of the braiding for the full Hopf bimodule.

\end{proof}

Now we can show the categories of full Yetter-Drinfeld modules and full Hopf bimodules are equivalent.

\begin{thm}\label{thm. equivalence of monoidal category}
    Let $(\cL, \cR, S)$ be a Hopf algebroid, then $(\mathcal{YD}^{\cR}_{\cR}, \ot_{A})$ is equivalent to $({}_{H}^{\cL}\CM_{H}^{\cR}, \ot_{H})$ as pre-braided monoidal category, and $(\mathcal{YD}^{\cH}_{\cR}, \ot_{A})$ is equivalent to $({}^{\cH}_{H}\mathcal{M}{}_{H}^{\cH}, \ot_{H})$ as braided monoidal category.
\end{thm}

\begin{proof}
    The first half of the proof can be given in the same way as \cite{schau2}. More precisely, for any $\Lambda^{1}, \Lambda^{2}\in \mathcal{YD}^{\cR}_{\cR}$ (resp. $\mathcal{YD}^{\cH}_{\cR}$),
    we can see $\xi: \cL\ot_{A}(\Lambda^{1}\ot_{A}\Lambda^{2})\to (\cL\ot_{A}\Lambda^{1})\ot_{H}(\cL\ot_{A}\Lambda^{2})$ given by
    \[\xi(X\ot_{A}\eta^{1}\ot_{A}\eta^{2})=X\ot_{A}\eta^{1}\ot_{H}1\ot_{A}\eta^{2}\]
    is invertible, with its inverse given by
    \[\xi^{-1}(X\ot_{A}\eta^{1}\ot_{H}Y\ot_{A}\eta^{2})=(X\ot_{A}\eta^{1})\ra Y\ot_{A}\eta^{2}.\]
    It is not hard to see $\xi$ is a morphism in the categories $({}^{\cL}_{H }\mathcal{M}{}_{ H}^{\cR}, \ot_{H})$ (resp. $({}^{\cH}_{H}\mathcal{M}{}_{H}^{\cH}, \ot_{H})$), and  satisfies the coherent condition. Now, for any $\Gamma^{1}, \Gamma^{2}\in {}^{\cL}_{H }\mathcal{M}{}_{ H}^{\cR}$ (resp.  ${}^{\cH}_{H}\mathcal{M}{}_{H}^{\cH}$), let's check $\varsigma: {}^{co\cL}\Gamma^{1}\ot_{A}{}^{co\cL}\Gamma^{2}\to {}^{co\cL}(\Gamma^{1}\ot_{H}\Gamma)$ given by
    \[\varsigma(\eta^{1}\ot_{A}\eta^{2})=\eta^{1}\ot_{H}\eta^{2}\]
 is invertible, with its inverse given by
 \[\varsigma^{-1}(\rho^{1}\ot_{H}\rho^{2})=\rho^{1}\ra \rho^{2}\mo{}_{+}\ot_{A}\rho^{2}\mo{}_{-}\la\rho^{2}\z.\]
Clearly, $\varsigma$ is well defined.    It is right $\cR$-linear:
\begin{align*}
    \varsigma((\rho^{1}\ot_{A}\rho^{2})\bra X)=&\varsigma(\rho^{1}\bra X\ro\ot_{A}\rho^{2}\bra X\rt)\\
    =&X\ro\m\la\rho^{1}\ra X\ro\p\ot_{H} X\rt\m\la\rho^{2}\ra X\rt\p\\
    =&X\m\la\rho^{1}\ra X\p\ro\ot_{H} X\p\rt\m\la\rho^{2}\ra X\p\rt\p\\
    =&X\m\la\rho^{1}\ot_{H} \rho^{2}\ra X\p\\
    =&\varsigma(\rho^{1}\ot_{A}\rho^{2})\bra X.
\end{align*}
It is trivial to see $\varsigma$ is right $\cR$-colinear (resp. $\cL$-colinear if $\Gamma^{1}, \Gamma^{2}\in {}^{\cH}_{H}\mathcal{M}{}_{H}^{\cH}$). And $\varsigma$ also satisfies the coherent condition. We can see $\varsigma^{-1}$ is well defined, since we know $\rho^{2}\mo{}_{+}\ot_{A}\rho^{2}\mo{}_{-}\la\rho^{2}\z\in \cL\ot_{A}{}^{co\cL}\Gamma^{2}$, we can also see
\begin{align*}
    {}_{L}\delta(\rho^{1}\ra\rho^{2}\mo{}_{+})\ot_{B}\rho^{2}\mo{}_{-}\la\rho^{2}\z=&\rho^{1}\mo\rho^{2}\mo{}_{+}\o\ot_{A}\rho^{1}\z\ra\rho^{2}\mo{}_{+}\t\ot_{A}\rho^{2}\mo{}_{-}\la\rho^{2}\z\\
    =&\rho^{1}\mo\rho^{2}\mo\o\ot_{B}\rho^{1}\z\ra\rho^{2}\mo\t{}_{+}\ot_{A}\rho^{2}\mo\t{}_{-}\la\rho^{2}\z\\
    =&1\ot_{B}\rho^{1}\ra\rho^{2}\mo{}_{+}\ot_{A}\rho^{2}\mo{}_{-}\la\rho^{2}\z,
\end{align*}
so the image of $\varsigma^{-1}$ belongs to ${}^{co\cL}\Gamma^{1}\ot_{A}{}^{co\cL}\Gamma^{2}$.
It is trivial to show that $\varsigma^{-1}$ is the inverse of $\varsigma$. Moreover, we have
\begin{align*}
    \varsigma\circ \sigma(\eta^{1}\ot_{A}\eta^{2})=\eta^{2}\rz\ot_{H}\eta^{1}\bra \eta^{2}\ro
    =\sigma\circ \varsigma(\eta^{1}\ot_{A}\eta^{2}),
\end{align*}
for any $\eta^{1}\ot_{A}\eta^{2}\in {}^{co\cL}\Gamma^{1}\ot_{A}{}^{co\cL}\Gamma^{2}$, where the $\sigma$ on the left-hand side is the braiding of the category of Yetter-Drinfeld modules (resp. full Yetter-Drinfeld modules) and the $\sigma$ on the right-hand side is the braiding of the category of Hopf bimodules (resp. full Hopf bimodules). So these monoidal categories are equivalent as pre-braided monoidal categories (resp. braided monoidal categories).

\end{proof}

Similarly,

\begin{thm}\label{thm. equivalence of monoidal category1}
    Let $(\cL, \cR, S)$ be a Hopf algebroid, then $({}^{\cL}_{\cL}\mathcal{YD}, \ot_{B})$ is a pre-braided monoidal category and $({}^{\cH}_{\cL}\mathcal{YD}, \ot_{B})$ is a braided monoidal category. Moreover, ${}({}_{\cL}^{\cL}\mathcal{YD}, \ot_{B})$ is equivalent to $({}_{H}^{\cL}\CM_{H}^{\cR}, \ot_{H})$ as pre-braided monoidal category, and $({}^{\cH}_{\cL}\mathcal{YD}, \ot_{B})$ is equivalent to $({}^{\cH}_{H}\mathcal{M}{}_{H}^{\cH}, \ot_{H})$ as braided monoidal category.
\end{thm}
\begin{proof}

    Given two Yetter-Drinfeld modules $\Lambda^{1}, \Lambda^{2}\in{}_{\cL}^{\cL}\mathcal{YD}$, and $\eta^{1}\in \Lambda^{1}$ and $\eta^{2}\in \Lambda^{2}$. The braiding is given by
   \[\sigma(\eta^{1}\ot_{B}\eta^{2})=\eta^{1}\mo\bla\eta^{2}\ot_{B}\eta^{1}\z.\]
   For $\Lambda^{1}, \Lambda^{2}\in {}^{\cH}_{\cL}\mathcal{YD}$, its inverse is
    \[\sigma^{-1}(\eta^{2}\ot_{B}\eta^{1})=\eta^{1}\rz\ot_{B}S^{-1}(\eta^{1}\rmo)\bla \eta^{2}.\]

\end{proof}

For anti-Hopf bimodules, left-right Yetter-Drinfeld modules, and right-left Yetter-Drinfeld modules, we have

\begin{thm}\label{thm. equivalence of monoidal category2}
    Let $(\cL, \cR, S)$ be a Hopf algebroid, then $({}^{\cR}\mathcal{YD}_{R}, \ot_{A})$ and $({}_{\cL}\mathcal{YD}^{\cL}, \ot_{B})$ are pre-braided monoidal categories,  $({}^{\cH}\mathcal{YD}_{\cR}, \ot_{A})$ and $({}_{\cL}\mathcal{YD}^{\cH}, \ot_{B})$ are braided monoidal categories. Moreover,  $({}^{\cR}\mathcal{YD}_{R}, \ot_{A})$ and $({}_{\cL}\mathcal{YD}^{\cL}, \ot_{B})$ are equivalent to $({}_{H}^{\cR}\CM_{H}^{\cL}, \ot^{H})$ as pre-braided monoidal categories, $({}^{\cH}\mathcal{YD}_{\cR}, \ot_{A})$ and $({}_{\cL}\mathcal{YD}^{\cH}, \ot_{B})$ are equivalent to $({}_{H}^{\cH}\CM_{H}^{\cH}, \ot^{H})$ as braided monoidal categories.
\end{thm}
\begin{proof}

    Given two Yetter-Drinfeld modules $\Lambda^{1}, \Lambda^{2}\in{}^{\cR}\mathcal{YD}_{\cR}$, and $\eta^{1}\in \Lambda^{1}$ and $\eta^{2}\in \Lambda^{2}$. The braiding is given by
   \[\tilde{\sigma}(\eta^{1}\ot_{A}\eta^{2})=\eta^{2}\bra \eta^{1}\rmo\ot_{A}\eta^{1}\rz,\]
  For $\Lambda^{1}, \Lambda^{2}\in {}^{\cH}\mathcal{YD}_{\cR}$, the inverse $\tilde{\sigma}^{-1}$ is
    \[\tilde{\sigma}^{-1}(\eta^{2}\ot_{A}\eta^{1})=\eta^{1}\z\ot_{A} \eta^{2}\bra S(\eta^{1}\mo).\]
It is not hard to see
\begin{align*}
    \tilde{\sigma}^{-1}\circ \tilde{\sigma}(\eta^{1}\ot_{A}\eta^{2})=&\eta^{1}\rz\z\ot_{A}\eta^{2}\bra (\eta^{1}\rmo S(\eta^{1}\rz\mo))\\
    =&\eta^{1}\z\ot_{A}\eta^{2}\bra (\eta^{1}\mo\ro S(\eta^{1}\mo\rt))\\
    =&\eta^{1}\ot_{A}\eta^{2},
\end{align*}
and
\begin{align*}
    \tilde{\sigma}\circ \tilde{\sigma}^{-1}(\eta^{2}\ot_{A}\eta^{1})=&\eta^{2}\bra(S(\eta^{1}\mo)\eta^{1}\z\rmo)\ot_{A}\eta^{1}\z\rz\\
    =&\eta^{2}\bra(S(\eta^{1}\rmo\o)\eta^{1}\rmo\o)\ot_{A}\eta^{1}\rz\\
    =&\eta^{2}\ot_{A}\eta^{1}.
\end{align*}
 Similarly, given two Yetter-Drinfeld modules $\Lambda^{1}, \Lambda^{2}\in{}_{\cL}\mathcal{YD}^{\cL}$, and $\eta^{1}\in \Lambda^{1}$ and $\eta^{2}\in \Lambda^{2}$, the braiding is given by
   \[\tilde{\sigma}(\eta^{1}\ot_{B}\eta^{2})=\eta^{2}\z\ot_{B}\eta^{2}\o\bla \eta^{1}.\]
   For $\Lambda^{1}, \Lambda^{2}\in {}_{\cL}\mathcal{YD}^{\cH}$, its inverse is
    \[\tilde{\sigma}^{-1}(\eta^{2}\ot_{B}\eta^{1})=S(\eta^{2}\ro)\bla \eta^{1}\ot_{B}\eta^{2}\rz.\]

\end{proof}

\begin{cor}
    Let $(\cL, \cR, S)$ be a Hopf algebroid, then $({}^{\cH}_{\cL}\mathcal{YD}, \ot_{B})$ and $(\mathcal{YD}^{\cH}_{\cR}, \ot_{A})$ are anti-monoidal equivalent to $({}^{\cH}{}\mathcal{YD}_{R}, \ot_{A})$ and $({}_{\cL}\mathcal{YD}^{\cH}, \ot_{B})$
as braided monoidal categories.
\end{cor}

\begin{proof}
     It is not hard to see $({}^{\cH}_{H}\mathcal{M}{}_{H}^{\cH}, \ot^{H})$ is the reverse monoidal category of $({}^{\cH}_{H}\mathcal{M}{}_{H}^{\cH}, \ot_{H})$ that preserve the braiding structure. More precisely,
    \[\sigma(\rho^{1}\ot_{H}\rho^{2})=\textrm{flip}\circ \tilde{\sigma}^{-1}(\rho^{2}\ot^{H}\rho^{1}).\]
    So we have the result.
\end{proof}

\section{Dual pairings between Yetter-Drinfeld modules}\label{sec5}

By \cite{schau1}, given two left bialgebroids, one can define a dual pairing.

\begin{defi}\label{def. dual pairing for left bialgeboids}
    Let $\cL$ and $\Pi$ be two left bialgebroids over $B$, a dual pairing between $\cL$ and $\Pi$ is a linear map $[\bullet |\bullet] :\cL\ot\Pi\to B$ such that:
    \begin{itemize}
        \item $[s_{L}(a)t_{L}(b)X s_{L}(c)t_{L}(d)| \alpha]f=a[X|s_{L}(c)t_{L}(f)\alpha s_{L}(d)t_{L}(b)]$,
        \item  $[X|\alpha\beta]=[X\o|\alpha t_{L}[X\t|\beta]]=[t_{L}[X\t|\beta]X\o|\alpha ]$,
        \item $[XY|\alpha]=[X s_{L}[Y|\alpha\o]|\alpha\t]=[X |s_{L}[Y|\alpha\o]\alpha\t]$,
        \item $[X|1]=\varepsilon(X)$,
        \item $[1|\alpha]=\varepsilon(\alpha)$,
    \end{itemize}
    for all $a, b, c, d, f\in B$, $\alpha, \beta\in \Pi$ and $X, Y\in \cL$.
\end{defi}

Similarly,

\begin{defi}\label{def. dual pairing for right bialgeboids}
     Let $\cR$ and $\Omega$ be two right bialgebroids over $A$, a dual pairing between $\Omega$ and $\cR$ is a linear map $(\bullet|\bullet):\Omega\ot \cR\to A$ such that
     \begin{itemize}
         \item $(s_{R}(a)t_{R}(b)\alpha s_{R}(c)t_{R}(d)| X)f=d(\alpha|s_{R}(c)t_{R}(a)X s_{R}(f)t_{R}(b))$,
         \item $(\alpha\beta|X)=(\beta|X\rt t_{R}(\alpha|X\ro))=(t_{R}(\alpha|X\ro)\beta|X\rt)$,
         \item $(\alpha|XY)=(\alpha\ro s_{R}(\alpha\rt|X)|Y)=(\alpha\ro |s_{R}(\alpha\rt|X)Y)$,
         \item $(\alpha|1)=\varepsilon(\alpha)$,
         \item $(1|X)=\varepsilon(X)$.
     \end{itemize}
\end{defi}

\begin{prop}
    Let $(\cL, \cR, S)$ and $(\Pi, \Omega, S)$ be Hopf algebroids, the existence of a dual pairing between $\cL$ and $\Pi$ is equivalent to the existence of a dual pairing between $\Omega$ and $\cR$.
\end{prop}
\begin{proof}
    Let $(\bullet|\bullet)$ be a dual pairing between $\Omega$ and $\cR$, then we can define
    \[[X|\alpha]:=(S^{-1}(\alpha)|S^{-1}(X)),\]
    for any $X\in \cL$ and $\alpha\in \Pi$.
    We can see
    \begin{align*}
        [s_{L}(f)t_{L}(d)Xs_{L}(b)t_{L}(a)|\alpha]\cdot_{B}c=&c\cdot_{A}(S^{-1}(\alpha)|s_{R}(b)t_{R}(a)S^{-1}(X)s_{R}(f)t_{R}(d))\\
        =&(s_{R}(a)t_{R}(d)S^{-1}(\alpha)s_{R}(b)t_{R}(c)|S^{-1}(X))\cdot_{A}f\\
        =&(S^{-1}(s_{L}(b)t_{L}(c)\alpha s_{L}(a)t_{L}(d))|S^{-1}(X))\cdot_{A}f\\
        =&f\cdot_{B}[X|s_{L}(b)t_{L}(c)\alpha s_{L}(a)t_{L}(d)].
    \end{align*}
    And
    \begin{align*}
        [X|\alpha\beta]=&(S^{-1}(\beta)S^{-1}(\alpha)|S^{-1}(X))\\
        =&(S^{-1}(\alpha)|S^{-1}(X)\rt t_{R}(S^{-1}(\beta)|S^{-1}(X)\ro))\\
        =&(S^{-1}(\alpha)|S^{-1}(t_{L}(S^{-1}(\beta)|S^{-1}(X\t)) X\o ))\\
        =&[t_{L}[X\t|\beta]X\o|\alpha].
    \end{align*}
    By the same method, we can get $[XY|\alpha]=[X s_{L}[Y|\alpha\o]|\alpha\t]=[X |s_{L}[Y|\alpha\o]\alpha\t]$. Moreover, $[X|1]=\varepsilon_{R}(S^{-1}(X))=\varepsilon_{L}(X)$, $[1|\alpha]=\varepsilon_{R}(S^{-1}(\alpha))=\varepsilon_{L}(\alpha)$. Similarly, if there is a dual pairing between the corresponding left bialgebroids, then we can define a dual pairing between $\Omega$ and $\cR$ by
    \[(\alpha|X):=[S(X)|S(\alpha)].\]
\end{proof}

\begin{defi}
    A dual pairing between two Hopf algebroids $(\cL, \cR, S)$ and $(\Pi, \Omega, S)$ is a dual pairing between the corresponding left bialgebroids $\cL$ and $\Pi$. Or equivalently, a dual pairing between two Hopf algebroids $(\cL, \cR, S)$ and $(\Pi, \Omega, S)$ is a dual pairing between the corresponding right bialgebroids $\Omega$ and $\cR$.
\end{defi}
For a left bialgebroid $\cL$ over $B$, we denote the collection of right $B$-module maps from $\cL$ to $B$ by $\cL^{\vee}$, namely, $\cL^{\vee}=\{\alpha:\cL\to B| \alpha(t_{L}(b)X)=\alpha(X)b, \forall X\in\cL, b\in B\}$. Similarly,  we denote the collection of left $B$-module maps from $\cL$ to $B$ by ${}^{\vee}\cL=\{\beta:\cL\to B| \beta(s_{L}(b)X)=b \beta(X), \forall X\in\cL, b\in B\}$. For a right bialgebroid $\cR$ over $A$, we denote the collection of right (resp. left) $A$-module maps from $\cR$ to $A$ by $\cR^{\vee}$ (resp. ${}^{\vee}\cR$). By \cite{schau1}, if $\cL$ is a left bialgebroid, then $\cL^{\vee}$ and ${}^{\vee}\cL$ are $B^{e}$ rings with the $B^{e}$-bimodule structure given by the dual pairing above. Similarly, $\cR^{\vee}$ and ${}^{\vee}\cR$ are $A^{e}$ ring for any right algebroid $\cR$. More precisely, let $\alpha\in \cR^{\vee}$ and $X\in\cR$, then
\[(s_{R}(a)t_{R}(b)\alpha s_{R}(c)t_{R}(d))( X)=d\alpha(s_{R}(c)t_{R}(a)X t_{R}(b)).\]

\begin{defi}
    A right  $B$-module $M$ is called a right  finitely generated projective $B$-module if there is
$\{x_{i}\}_{i}\in M$ and $\{x^{i}\}_{i}\in M^{\vee}$  (where $M^{\vee}$ is the collection of right module maps from
$M$ to $B$ equipped with the left $B$-module structure given by $(b.f)(m)=bf(m)$ for any $f\in M^{\vee}$, $m\in M$ and $b\in B$), such that for any $X\in M$, $X=x_{i}.(x^{i}(X))$. A left  $B$-module $N$ is called a left  finitely generated projective $B$-module if  there is
$\{y_{i}\}_{i}\in N$ and $\{y^{i}\}_{i}\in {}^{\vee}N$ (the collection of left module map from $N$ to $B$), such that for any $Y\in N$, $Y=(y^{i}(Y)).y_{i}$.
\end{defi}

By assuming the $B$-bimodule structures of left bialgebroids given by (\ref{eq:rbgd.bimod}) and the $A$-bimodule structures of right bialgebroids given by (\ref{eq:rbgd.bimod1}), then by \cite{schau1,schau3}, we have

\begin{thm}
    Let $\Pi$ (resp. $\cL$) be a right (resp. left) finitely generated projective left bialgebroid, then $\Pi^{\vee}$ (resp. ${}^{\vee}\cL$) is a left (resp. right)  finitely generated projective left bialgebroid.

    Let $\cR$ (resp. $\Omega$) be a right (resp. left) finitely generated projective right bialgebroid, then $\cR^{\vee}$ (resp. ${}^{\vee}\Omega$) is a left (resp. right) finitely generated projective right bialgebroid.
\end{thm}

The left (resp. right) bialgebroid structures on $\Pi^{\vee}$,  ${}^{\vee}\cL$, $\cR^{\vee}$ and ${}^{\vee}\Omega$ can be given by Definitions \ref{def. dual pairing for left bialgeboids} and \ref{def. dual pairing for right bialgeboids}. For example, the product of $\cR^{\vee}$ is given by
\[\alpha\beta(X)=\beta(X\rt t_{R}(\alpha(X\ro))),\]
and the coproduct on $\cR^{\vee}$ is determined by

\[\alpha(XY)=\alpha\ro (s_{R}(\alpha\rt(X))Y),\]
for any $X, Y\in\cR$ and $\alpha, \beta\in\cR^{\vee}$. In the following, we will always use notation $(\alpha|X):=\alpha(X)$, for any $X\in \cR$ and $\alpha\in \cR^{\vee}$.

\begin{prop}
    Let $(\cL, \cR, S)$ be a Hopf algebroid. Then $\cR$ is a right (resp. left) finitely generated projective $A$-module if and only if $\cL$ is a left (resp. right) finitely generated projective $A$-module.
\end{prop}

\begin{proof}
    As $\cR$ is a right finitely generated projective $A$-module, then there is
$\{x_{i}\}_{i}\in \cR$ and $\{x^{i}\}_{i}\in \cR^{\vee}$ , such that for any $X\in \cR$, $X=x_{i}s_{R}(x^{i}(X))$.
We can define a map $\phi: \cR^{\vee}\to {}^{\vee}\cL$ by
\[[X|\phi(\alpha)]:=(\alpha|S^{-1}(X)),\]
for any $X\in\cR$, $\alpha\in\cR^{\vee}$. Indeed,
\begin{align*}
    [s_{L}(b)X|\phi(\alpha)]=(\alpha|S^{-1}(X)s_{R}(b))=b\cdot_{B}[X|\alpha].
\end{align*}

Define $y^{i}:=\phi(x^{i})$ and $y_{i}:=S(x_{i})$, then we can see
\begin{align*}
    s_{L}[X|y^{i}]y_{i}=s_{L}(x^{i}|S^{-1}(X))y_{i}=S(x_{i}s_{R}(x^{i}|S^{-1}(X)))=X.
\end{align*}

The rest of the statement can be similarly proved.

\end{proof}


\begin{defi}
    Let $\cR$ and $\Omega$ be right bialgebroids, and $\Lambda\in\mathcal{YD}^{\cR}_{\cR}$, $T\in{}^{\Omega}\mathcal{YD}_{\Omega}$, a dual pairing between $T$ and $\Lambda$ is a linear map $(\bullet,\bullet): T\ot \Lambda\to A$ such that:
    \begin{itemize}
        \item There is a dual pairing $(\bullet|\bullet)$ between $\Omega$ and $\cR$,
        \item $(f\bra (s_{R}(a)t_{R}(b)), \eta)c=b(f, \eta\bra (s_{R}(c)t_{R}(a)))$,
        \item $(f\bra \alpha, \eta)=(t_{R}(f, \eta\rz)\alpha|\eta\ro)=(\alpha|\eta\ro t_{R}(f, \eta\rz))$,
        \item $(f, \eta\bra X)=(f\rmo|s_{R}(f\rz,\eta)X)=(f\rmo s_{R}(f\rz,\eta)|X)$,
    \end{itemize}
\end{defi}

\begin{thm}\label{thm. equivalent of left and right Yetter-Drinfeld modules}
Let $\cR$
be a right bialgebroid and $\Lambda\in\mathcal{YD}^{\cR}_{\cR}$, if $\cR$ and $\Lambda$ are both right finitely generated projective $A$-module, then $\Lambda^{\vee}\in {}^{\cR^{\vee}}\mathcal{YD}_{\cR^{\vee}}$, and it is a left finitely generated projective $A$-module.\\
Similarly, let $\Omega$
be a right bialgebroid and $T\in {}^{\Omega}\mathcal{YD}_{\Omega}$, if $\Omega$ and $T$ are left finitely generated projective $A$-module, then ${}^{\vee}T\in \mathcal{YD}^{{}^{\vee}\Omega}_{{}^{\vee}\Omega}$, and it is a right finitely generated projective $A$-module.\\
Moreover, $\Lambda\simeq {}^{\vee}(\Lambda^{\vee})$ and $T\simeq ({}^{\vee}T)^{\vee}$ as Yetter-Drinfeld modules.
\end{thm}

\begin{proof}
    First, we can define the $A$-bimodule structure on $\Lambda^{\vee}$ by
    \[(a^{.}f^{.}a', \eta)=a(f, \eta\bra t_{R}(a')),\]
    for any $f\in \Lambda^{\vee}$, $\eta\in \Lambda$ and $a, a'\in A$, and we use notation $(f, \eta):=f(\eta)$.

    Second, we can define a map $\phi: \Lambda^{\vee}\ot_{A^{e}}\cR^{\vee}\to (\Lambda\times_{A}\cR)^{\vee}$ by
    \[\phi(f\ot_{A^{e}}\alpha)(\eta\ot_{A}X)=(t_{R}(f, \eta)\alpha|X)=(\alpha|Xt_{R}(f, \eta)),\]

    where $\Lambda\times_{A}\cR$ is the Takeuchi product, and the balanced tensor product $\ot_{A^{e}}$ reads $a^{.}f^{.}a'\ot_{A^{e}}\alpha=f\ot_{A^{e}}s_{R}(a')t_{R}(a)\alpha$.  Here we only show
    \[(t_{R}(f^{.}a, \eta)\alpha|X)=(t_{R}(f, \eta\bra t_{R}(a))\alpha|X)=(t_{R}(f, \eta )\alpha| t_{R}(a)X)=(t_{R}(f, \eta )s_{R}(a)\alpha|X ).\]
    It is not hard to see $\phi$ factors through all the rest balanced tensor products.
With the help of $\phi$, we can define the right action of $\cR^{\vee}$ on $\Lambda^{\vee}$ by
\begin{align*}
    (f\bra \alpha, \eta)=(t_{R}(f, \eta\rz)\alpha|\eta\ro)=(\alpha|\eta\ro t_{R}(f, \eta\rz)),
\end{align*}
it is not hard to check $f\bra(\alpha\beta)=(f\bra\alpha)\bra\beta$ since $\Lambda$ is a right $\cR$-comodule. We can also see $a^{.}f=f\bra t_{R}(a)$ and $f^{.}a=f\bra s_{R}(a)$:
\begin{align*}
    (f\bra s_{R}(a), \eta)=(f, \eta\rz)\varepsilon(t_{R}(a)\eta\ro)=(f, \eta\rz\bra t_{R}(a))\varepsilon(\eta\ro)=(f, \eta\bra t_{R}(a)),
\end{align*}
and
\begin{align*}
    (f\bra t_{R}(a), \eta)=a(f, \eta\rz)\varepsilon(\eta\ro)=a(f, \eta).
\end{align*}

Third, we can define a map: $\psi: \cR^{\vee}\times_{A}\Lambda^{\vee}\to (\Lambda\ot_{A^{e}}\cR)^{\vee}$ by
\begin{align*}
    \psi(\alpha\ot_{A}f)(\eta\ot_{A^{e}}X)=(\alpha|s_{R}(f,\eta)X)=(\alpha s_{R}(f,\eta)|X).
\end{align*}
Here we only show
\[(\alpha s_{R}(f,\eta\bra t_{R}(a))|X)=(\alpha s_{R}(f\bra s_{R}(a),\eta)|X)=(s_{R}(a)\alpha s_{R}(f,\eta)|X)=(\alpha s_{R}(f,\eta)|t_{R}(a)X).\]
It is not hard to see $\psi$ factor through all the rest balanced tensor products. As $\Lambda$ is a right finitely generated projective $A$-module, then for any $\eta\in \Lambda$, there is $\lambda_{i}\ot_{A}\lambda^{i}\in \Lambda\ot_{A}\Lambda^{\vee}$, such that for any $\eta\in \Lambda$, we have $\eta=\lambda_{i}\bra s_{R}(\lambda^{i},\eta)$.
We can see for any $f\in\Lambda^{\vee}$, $f=\lambda^{i}\bra t_{R}(f, \lambda_{i})$:
\[(\lambda^{i}\bra t_{R}(f, \lambda_{i}), \eta)=(\lambda^{i}, \eta\bra s_{R}(f, \lambda_{i}))=(f, \eta).\]
Moreover, $\lambda_{i}\ot_{A}\lambda^{i}$ satisfies $\lambda_{i} \bra t_{R}(a)\ot_{A}\lambda^{i}=\lambda_{i} \ot_{A}\lambda^{i}\bra s_{R}(a)$. Indeed, on the one hand,
\[\eta\bra t_{R}(a)=\lambda_{i}\bra s_{R}(\lambda^{i},\eta\bra t_{R}(a))=\lambda_{i}\bra s_{R}(\lambda^{i}\bra s_{R}(a),\eta),\]
on the other hand,

\[\eta\bra t_{R}(a)=(\lambda_{i}\bra t_{R}(a))\bra s_{R}(\lambda^{i},\eta).\]
Now, we can show $\psi$ is bijective, with its inverse given by
\begin{align*}
    \psi^{-1}(F):=F(\lambda_{i}\ot_{A^{e}}\bullet)\ot_{A}\lambda^{i},
\end{align*}
for any $F\in (\Lambda\ot_{A^{e}}\cR)^{\vee}$. We first observe that the image of $\psi^{-1}$ belongs to Takeuchi product:

\begin{align*}
   s_{R}(a)F(\lambda_{i}\ot_{A^{e}}\bullet)\ot_{A}\lambda^{i}=&F(\lambda_{i}\ot_{A^{e}}t_{R}(a)\bullet)\ot_{A}\lambda^{i} \\
   =&F(\lambda_{i}\bra t_{R}(a)\ot_{A^{e}}\bullet)\ot_{A}\lambda^{i} \\
   =&F(\lambda_{i}(a)\ot_{A^{e}}\bullet)\ot_{A}\lambda^{i} \bra s_{R}(a).
\end{align*}
 We can also see
 \begin{align*}
     \psi\circ\psi^{-1}(F)(\eta\ot_{A^{e}}X)=F(\lambda_{i}\ot_{A^{e}}s_{R}(\lambda^{i}, \eta)X)=F(\eta\ot_{A^{e}}X),
 \end{align*}
 and
 \begin{align*}
     \psi^{-1}\circ\psi(\alpha\ot_{A}f)=&\psi(\alpha\ot_{A}f)(\lambda_{i}\ot_{A^{e}}\bullet)\ot_{A}\lambda^{i}\\
     =&(\alpha s_{R}(f, \lambda_{i})|\bullet)\ot_{A}\lambda^{i}\\
     =&(\alpha |\bullet)\ot_{A}\lambda^{i}\bra t_{R}(f, \lambda_{i})\\
     =&\alpha\ot_{A}f.
 \end{align*}
For later use, we can define another bijective map $\tilde{\psi}: \cR^{\vee}\ot_{A}\Lambda^{\vee}\to (\Lambda\ot_{A_{s}}\cR)^{\vee}$ by
 \[\tilde{\psi}(\alpha\ot_{A}f)(\eta\ot_{A_{s}}X)=(\alpha|s_{R}(f,\eta)X)=(\alpha s_{R}(f,\eta)|X),\]
  where $\ot_{A_{s}}$ is the balanced tensor product induced by $s_{R}$.
 Let $f\in \Lambda^{\vee}$, define $\tilde{f}\in (\Lambda\ot_{A^{e}}\cR)^{\vee}$ by
\begin{align*}
    \tilde{f}(\eta\ot_{A^{e}}X)=f(\eta\bra X).
\end{align*}
The coaction of $f$ can be given through $\psi$ in the following
\begin{align}
   {}_{R}\delta(\eta):=\psi^{-1}(\tilde{f}).
\end{align}
It is not hard to see ${}_{R}\delta$ is a $A$-bilinear map.
We can also check
\begin{align*}
(f\rmo s_{R}(f\rz\rmo s_{R}(f\rz\rz, \eta)|X)|Y)=&(f, (\eta\bra X)\bra Y)
    =(f, \eta\bra(XY))\\
        =&(f\rmo\ro s_{R}(f\rmo\rt s_{R}(f\rz, \eta)|X)|Y),
\end{align*}
so ${}_{R}\delta$ is a well defined coaction. Finally, let's check the Yetter-Drinfeld condition. On the one hand
\begin{align*}
    \tilde{\psi}(f\rmo\alpha\ro\ot_{A}f\rz\bra\alpha\rt)(\eta\ot_{A_{s}}X)=&(f\rmo\alpha\ro|s_{R}(f\rz\bra\alpha\rt, \eta)X)\\
    =&(f\rmo\alpha\ro|s_{R}(\alpha\rt|\eta\ro t_{R}(f\rz, \eta\rz))X)\\
    =&(f\rmo s_{R}(f\rz, \eta\rz)\alpha\ro|s_{R}(\alpha\rt|\eta\ro )X)\\
    =&(\alpha\ro|s_{R}(\alpha\rt|\eta\ro )X\rt t_{R}(f\rmo s_{R}(f\rz, \eta\rz)|X\ro))\\
    =&(\alpha|\eta\ro X\rt t_{R}(f, \eta\rz\bra X\ro)),
\end{align*}
on the other hand
\begin{align*}
    \tilde{\psi}(\alpha\rt (f\bra& \alpha\ro)\rmo\ot_{A}(f\bra \alpha\ro)\rz)(\eta\ot_{A_{s}}X)\\
    =&(\alpha\rt (f\bra \alpha\ro)\rmo s_{R}((f\bra \alpha\ro)\rz, \eta)|X)\\
    =&( (f\bra \alpha\ro)\rmo s_{R}((f\bra \alpha\ro)\rz, \eta)|X\rt t_{R}(\alpha\rt|X\ro))\\
    =&(f\bra \alpha\ro, \eta\bra (X\rt t_{R}(\alpha\rt|X\ro)) )\\
    =&(f\bra (\alpha\ro s_{R}(\alpha\rt|X\ro)), \eta\bra X\rt )\\
    =&(\alpha\ro s_{R}(\alpha\rt|X\ro)| (\eta\bra X\rt)\ro t_{R}(f, (\eta\bra X\rt)\rz))\\
    =&(\alpha|X\ro (\eta\bra X\rt)\ro t_{R}(f, (\eta\bra X\rt)\rz)),
\end{align*}
as $\Lambda$ is a right-right Yetter-Drinfeld module, we get the left-right Yetter-Drinfeld condition for $\Lambda^{\vee}$. The second half of the statement can be similarly proved. And $\Lambda\simeq {}^{\vee}(\Lambda^{\vee})$ and $T\simeq ({}^{\vee}T)^{\vee}$ can be given in its canonical way, as $\Lambda$ is right finitely generated projective $A$-module and $T$ is left finitely generated projective $A$-module.
\end{proof}

\begin{prop}
    Let $\cR$ and $\Omega$ be right bialgebroids, and $\Lambda^{i}\in\mathcal{YD}^{\cR}_{\cR}$, $T^{i}\in{}^{\Omega}\mathcal{YD}_{\Omega}$, such that there is a dual pairing between $T^{i}$ and $\Lambda^{i}$, for $i=1,2$, then the map
    $(\bullet,\bullet)_{2}: (T^{2}\ot_{A}T^{1})\ot (\Lambda^{1}\ot_{A}\Lambda^{2})\to A$ given by
    \begin{align*}
        (f\ot_{A}g, \eta^{1}\ot_{A}\eta^{2})_{2}=(f\bra s_{R}(g, \eta^{1}), \eta^{2})=(f, \eta^{2}\bra t_{R}(g, \eta^{1}))
    \end{align*}
       for any $f\ot_{A}g\in T^{2}\ot_{A}T^{1}$ and $\eta^{1}\ot_{A}\eta^{2}\in \Lambda^{1}\ot_{A}\Lambda^{2}$
    is a dual-pairing between Yetter-Drinfeld modules.
\end{prop}
\begin{proof}
    It is not hard to see
    \[((f\ot_{A}g)\bra (s_{R}(a)t_{R}(b)), (\eta^{1}\ot_{A}\eta^{2}))_{2}c=b((f\ot_{A}g), (\eta^{1}\ot_{A}\eta^{2})\bra (s_{R}(c)t_{R}(a)))_{2},\]
   for any $a, b, c\in A$. We can also see
    \begin{align*}
        ((f\ot_{A}g)\bra\alpha, \eta^{1}\ot_{A}\eta^{2})_{2}=&(f\bra\alpha\ro\ot_{A}g\bra\alpha\rt, \eta^{1}\ot_{A}\eta^{2})_{2}\\
        =&(f\bra\alpha\ro, \eta^{2}\bra t_{R}(g\bra \alpha\rt, \eta^{1}))\\
        =&(t_{R}(f, \eta^{2}\rz)\alpha\ro|s_{R}(t_{R}(g, \eta^{1}\rz)\alpha\rt|\eta^{1}\ro)\eta^{2}\ro)\\
        =&(t_{R}(f, \eta^{2}\rz)\alpha|\eta^{1}\ro t_{R}(g, \eta^{1}\rz)\eta^{2}\ro)\\
        =&(t_{R}(f, \eta^{2}\rz\bra t_{R}(g, \eta^{1}\rz))\alpha|\eta^{1}\ro \eta^{2}\ro)\\
        =&(t_{R}(f, \eta^{2}\rz\bra t_{R}(g, \eta^{1}\rz))\alpha|\eta^{1}\ro \eta^{2}\ro)\\
        =&(t_{R}(f\ot_{A}g, \eta^{1}\rz\ot_{A}\eta^{1}\rz)_{2}\alpha|\eta^{1}\ro \eta^{2}\ro),
    \end{align*}
and by the same method, we have
\begin{align*}
    (f\ot_{A}g, (\eta^{1}\ot_{A}\eta^{2})\bra X)_{2}=(g\rmo f\rmo|s_{R}(f\rz\ot_{A}g\rz, \eta^{1}\ot_{A}\eta^{2})_{2}X).
\end{align*}

\end{proof}

\begin{thm}
    Let $\cR$ be a right finitely generated projective right bialgebroid and $\Omega$ be a left finitely generated projective right bialgebroid such that $\Omega=\cR^{\vee}$ (or $\cR={}^{\vee}\Omega$), then the category of right finitely generated projective right-right Yetter-Drinfeld modules of $\cR$ is anti-monoidal equivalent to the category of left finitely generated projective left-right Yetter-Drinfeld modules of $\Omega$ as pre-braided monoidal category.

\end{thm}

\begin{proof}
    First, let's show the category right finitely generated projective right-right Yetter-Drinfeld modules of $\cR$ is monoidal, given two right finitely generated projective right-right Yetter-Drinfeld modules $\Lambda^{1}$ and $\Lambda^{2}$ of $\cR$, we know for any $\eta^{1}\in \Lambda^{1}$ and $\eta^{2}\in \Lambda^{2}$, there is $\lambda_{i}\ot_{A}\lambda^{i}\in \Lambda^{1}\ot_{A}\Lambda^{1\vee}$ and $\xi_{j}\ot_{A}\xi^{j}\in \Lambda^{2}\ot_{A}\Lambda^{2\vee}$, such that $\eta^{1}=\lambda_{i}\bra s_{R}(\lambda^{i},\eta^{2})$ and $\eta^{2}=\xi_{j}\bra s_{R}(\xi^{j},\eta^{2})$. We can define a map $\psi: \Lambda^{2\vee}\ot_{A}\Lambda^{1\vee}\to (\Lambda^{1}\ot_{A}\Lambda^{2})^{\vee}$ by
    \begin{align*}
        \psi(f\ot_{A}g)(\eta^{1}\ot_{A}\eta^{2}):=(f\ot_{A}g, \eta^{1}\ot_{A}\eta^{2})_{2}=(f\bra s_{R}(g, \eta^{1}), \eta^{2})=(f, \eta^{2}\bra t_{R}(g, \eta^{1})).
    \end{align*}
As $\Lambda^{1}$ and $\Lambda^{2}$ are right finitely generated projective, by the same method as in  Theorem \ref{thm. equivalent of left and right Yetter-Drinfeld modules} we can see $\psi$ is bijective with its inverse given by
\begin{align*}
    \psi^{-1}(F):=F(\lambda_{i}\ot_{A}\bullet)\ot_{A}\lambda^{i}.
\end{align*}
Thus, we can see
\begin{align*}
    (\lambda_{i}\ot_{A}\xi_{j})\bra s_{R}(\xi^{j}\ot_{A}\lambda^{i}, \eta^{1}\ot_{A}\eta^{2})_{2}=&(\lambda_{i}\ot_{A}\xi_{j})\bra s_{R}(\xi^{j} \bra s_{R}(\lambda^{i}, \eta^{1}),\eta^{2})\\
    =&(\lambda_{i}\ot_{A}\xi_{j}\bra t_{R}(\lambda^{i}, \eta^{1}))\bra s_{R}(\xi^{j} ,\eta^{2})\\
    =&(\lambda_{i}\bra s_{R}(\lambda^{i}, \eta^{1})\ot_{A}\xi_{j})\bra s_{R}(\xi^{j} ,\eta^{2})\\
    =&\eta^{1}\ot_{A}\eta^{2},
\end{align*}
where we use the fact $\xi_{i}\ot_{A}\xi^{i}$ satisfies $\xi_{i} \bra t_{R}(a)\ot_{A}\xi^{i}=\xi_{i} \ot_{A}\xi^{i}\bra s_{R}(a)$ which is given by the proof of Theorem \ref{thm. equivalent of left and right Yetter-Drinfeld modules}. So the category right finitely generated projective right-right Yetter-Drinfeld modules of $\cR$ is monoidal. Moreover, we can see $\psi$ satisfies the coherent condition, on the one hand,
\begin{align*}
    ((\psi)\circ(\psi\ot_{A}\id)(f\ot_{A}g\ot_{A}h), \eta^{1}\ot_{A}\eta^{2}\ot_{A}\eta^{3})=&(\psi(f\ot_{A}g),(\eta^{2}\ot_{A}\eta^{3}) \bra t_{R}(h, \eta^{1}))_{2}\\
    =&(f, \eta^{3}\bra t_{R}(g, \eta^{2}\bra t_{R}(h, \eta^{1}))),
\end{align*}
on the other hand,
\begin{align*}
    ((\psi)\circ(\id\ot_{A}\psi)(f\ot_{A}g\ot_{A}h), \eta^{1}\ot_{A}\eta^{2}\ot_{A}\eta^{3})=&(f, \eta^{3}\bra t_{R}(\psi(g\ot_{A}h), \eta^{1}\ot_{A}\eta^{2})_{2})\\
    =&(f, \eta^{3}\bra t_{R}(g, \eta^{2}\bra t_{R}(h, \eta^{1}))).
\end{align*}

Similarly, the category left finitely generated projective left-right Yetter-Drinfeld modules of $\Omega$ is monoidal as well. More precisely, as $T^{1}$ and $T^{2}$ are left-right Yetter-Drinfeld modules, as they are left finitely generated projective, then for any $g\in T^{1}$, $f\in T^{2}$, there is $\lambda_{i}\ot_{A}\lambda^{i}\in {}^{\vee}T^{1}\ot_{A}T^{1}$ and $\xi_{j}\ot_{A}\xi^{j}\in {}^{\vee}T^{2}\ot_{A}T^{2}$, such that $f=\xi^{i}\bra t_{R}(f, \xi_{i})$ and $g=\lambda^{i}\bra t_{R}(g, \lambda_{i})$. There is also a bijective map $\phi: {}^{\vee}T^{1}\ot_{A}{}^{\vee}T^{2}\to {}^{\vee}(T^{2}\ot_{A}T^{1})$given by
\[\phi(\eta^{1}\ot_{A}\eta^{2})(f\ot_{A}g)=(f, \eta^{2} \bra t_{R}(g, \eta^{1}))=(f\bra s_{R}(g, \eta^{1}), \eta^{2}).\]
By Theorem \ref{thm. equivalent of left and right Yetter-Drinfeld modules}, we can see these two monoidal categories are equivalent. Recall that, the braiding for $\mathcal{YD}^{\cR}_{\cR}$ and ${}^{\Omega}\mathcal{YD}_{\Omega}$ are given by:
\[\sigma(\eta^{1}\ot_{A}\eta^{2})=\eta^{2}\rz\ot_{A}\eta^{1}\bra\eta^{2}\ro,\]
for $\eta^{1}\ot_{A}\eta^{2}\in \Lambda^{1}\ot_{A}\Lambda^{2}$ and $\Lambda^{1}, \Lambda^{2}\in\mathcal{YD}^{\cR}_{\cR}$,
\[ \tilde{\sigma}(f\ot_{A}g)=g\bra f\rmo\ot_{A}f\rz,\]
for $f\ot_{A}g\in T^{2}\ot_{A}T^{1}$ and $T^{1}, T^{2}\in{}^{\Omega}\mathcal{YD}_{\Omega}$. We can see
\begin{align*}
    (\psi\circ\tilde{\sigma}(f\ot_{A}g), \eta^{2}\ot_{A}\eta^{1})=&(g\bra f\rmo\ot_{A}f\rz, \eta^{2}\ot_{A}\eta^{1})_{2}\\
    =&(g\bra (f\rmo s_{R}(f\rz, \eta^{2})), \eta^{1})\\
    =&(f\rmo s_{R}(f\rz, \eta^{2})|\eta^{1}\ro t_{R}(g, \eta^{1}\rz))\\
    =&(f, (\eta^{2}\bra\eta^{1}\ro)\bra t_{R}(g, \eta^{1}\rz))\\
    =&(f\ot_{A}g, \eta^{1}\rz\ot_{A}\eta^{2}\bra \eta^{1}\ro)_{2}\\
    =&(f\ot_{A}g, \sigma(\eta^{2}\ot_{A}\eta^{1}))_{2}\\
    =&(\sigma^{\vee}\circ\psi(f\ot_{A}g), \eta^{2}\ot_{A}\eta^{1}),
\end{align*}
where $\sigma^{\vee}$ is given by
\[(\sigma^{\vee}(T), \eta^{2}\ot_{A}\eta^{1}):=(T, \sigma(\eta^{2}\ot_{A}\eta^{1})),\]
    for any $T\in (\Lambda^{1}\ot_{A}\Lambda^{2})^{\vee}$. So the right (resp. left) dual functor is a  braided monoidal functor.
\end{proof}

Recall that if $\Lambda\in\mathcal{YD}^{\cR}_{\cR}$ and right finitely generated projective, then there is $\{\lambda_{i}\}_{i}\in \Lambda$ and $\{\lambda^{i}\}_{i}\in \Lambda^{\vee}$, such that $\eta=\lambda_{i}\bra s_{R}(\lambda^{i},\eta)$, and  we have the following proposition:
\begin{prop}
    Let $\cR$ be a right bialgebroid and $\Lambda\in\mathcal{YD}^{\cR}_{\cR}$, such that both $\cR$ and $\Lambda$ are right finitely generated projective. Define  $F^{j}_{i}\in\cR^{\vee}$ by $(F^{j}_{i}|X):=(\lambda^{j}, \lambda_{i}\bra X)$ for any $X\in\cR$. As both $\cR$ and $\Lambda$ are right finitely generated projective, $\delta_{R}(\lambda_{i})=\sum_{k}\lambda_{k}\ot_{A}v^{k}_{i}$, for some $\{v^{k}_{i}\}_{k}\in \cR$, then we have
    \begin{itemize}
        \item [(1)] $\lambda_{i}=\lambda_{j}\bra s_{R}(F^{j}_{i}|1)$,
        \item[(2)] $\Delta_{\cR^{\vee}}(F^{i}_{j})=F^{i}_{k}\ot_{A}F^{k}_{j}$,
        \item[(3)] $\Delta_{R}(v^{k}_{j})=v^{k}_{i}\ot_{A}v^{i}_{j}$,
        \item[(4)] $v^{k}_{i}X\rt t_{R}(F^{j}_{k}|X\ro)=X\ro v^{j}_{k} s_{R}(F^{k}_{i}|X\rt)$,
        \item[(5)] $\sigma(\lambda_{i}\ot_{A}\lambda_{j})=\lambda_{k}\ot_{A}\lambda_{l} \bra s_{R}(F^{l}_{i}|v^{k}_{j})$,
    \end{itemize}
    for any $X, Y\in \cR$.
\end{prop}

\begin{proof}
    By definition, it is not hard to see
    \[\lambda_{i}\bra X=\lambda_{j}\bra s_{R}(\lambda^{j}, \lambda_{i}\bra X)=\lambda_{j}\bra s_{R}(F^{j}_{i}|X),\]
for any $X\in\cR$.
As a result, $\lambda_{i}=\lambda_{j}\bra s_{R}(F^{j}_{i}|1)$.

For (2), we first observe that
$(\lambda^{i}, \lambda_{j})(F^{j}_{k}|X)=(\lambda^{i}, \lambda_{j}\bra s_{R}(F^{j}_{k}|X))=(\lambda^{i}, \lambda_{j}\bra s_{R}(\lambda^{j}, \lambda_{k}\bra X))=(\lambda^{i}, \lambda_{k}\bra X)=(F^{i}_{k}|X)$. Thus we have

\begin{align*}
    (F^{j}_{i}|XY)=&(\lambda^{j}, (\lambda_{i}\bra X)\bra Y)=(\lambda^{j}, (\lambda_{k}\bra s_{R}(F^{k}_{i}|X))\bra Y)=(\lambda^{j}, \lambda_{k} \bra(s_{R}(F^{k}_{i}|X)Y))\\
    =&(\lambda^{j}, \lambda_{l}\bra s_{R}(F^{l}_{k}|s_{R}(F^{k}_{i}|X)Y))=(\lambda^{j}, \lambda_{l})(F^{l}_{k}|s_{R}(F^{k}_{i}|X)Y)\\
    =&(F^{j}_{k}|s_{R}(F^{k}_{i}|X)Y).
\end{align*}
As a right $A$-linear map of $\cR$, we can see $\Delta_{\cR^{\vee}}(F^{i}_{j})=F^{i}_{k}\ot_{A}F^{k}_{j}$.

 For (3), we have
\begin{align*}
\lambda_{k}\ot_{A}\Delta_{R}(v^{k}_{j})=(\id\ot_{A}\Delta_{R})\circ \delta_{R}(\lambda_{j})=(\delta_{R}\ot_{A}\id)\circ \delta_{R}(\lambda_{j})=\lambda_{k}\ot_{A}v^{k}_{i}\ot_{A}v^{i}_{j}.
\end{align*}

For (4), we have on the one hand,
\begin{align*}
    (\lambda_{i})\rz\bra X\ro\ot_{A}(\lambda_{i})\ro X\rt=&\lambda_{k}\bra X\ro\ot_{A}v^{k}_{i}X\rt=\lambda_{j}\bra s_{R}(F^{j}_{k}|X\ro)\ot_{A}v^{k}_{i}X\rt\\
    =&\lambda_{j}\ot_{A}v^{k}_{i}X\rt t_{R}(F^{j}_{k}|X\ro),
\end{align*}
on the other hand,
\begin{align*}
    (\lambda_{i}\bra X\rt)\rz\ot_{A}X\ro  (\lambda_{i}\bra X\rt)\ro=\lambda_{j}\ot_{A}X\ro v^{j}_{k}s_{R}(F^{k}_{i}|X\rt),
\end{align*}
so we have (4) by Yetter-Drinfeld condition.

For (5), we have
\begin{align*}
    \sigma(\lambda_{i}\ot_{A}\lambda_{j})=\lambda_{j}{}\rz\ot_{A}\lambda_{i}\bra \lambda_{j}{}\ro=\lambda_{k}\ot_{A}\lambda_{l} \bra s_{R}(F^{l}_{i}|v^{k}_{j}).
\end{align*}


\end{proof}

Similarly, if $T\in{}^{\Omega}\mathcal{YD}_{\Omega}$ and left finitely generated projective, then there is $\{\lambda_{i}\}_{i}\in {}^{\vee}T$ and $\{\lambda^{i}\}_{i}\in T$, such that $f=\lambda^{i}\bra t_{R}(f,\lambda_{i})$, and  we have the following proposition:
\begin{prop}
    Let $\Omega$ be a right bialgebroid and $T\in{}^{\Omega}\mathcal{YD}_{\Omega}$, such that both $\Omega$ and $T$ are left finitely generated projective. Define  $V^{j}_{i}\in{}^{\vee}\Omega$ by $(\alpha|V^{j}_{i}):=(\lambda^{j}\bra\alpha, \lambda_{i})$ for any $\alpha\in\Omega$. As both $\Omega$ and $T$ are left finitely generated projective, ${}_{R}\delta(\lambda_{i})=\sum_{k}f^{i}_{k}\ot_{A}\lambda^{k}$, for some $\{f^{i}_{k}\}_{k}\in \Omega$, then we have
    \begin{itemize}
        \item [(1)] $\lambda_{i}=\lambda^{j}\bra t_{R}(1|V^{j}_{i})$,
        \item[(2)] $\Delta_{{}^{\vee}\Omega}(V^{i}_{k})=V^{i}_{j}\ot_{A}V^{j}_{k}$,
        \item[(3)] $\Delta_{\Omega}(f^{k}_{j})=f^{k}_{i}\ot_{A}f^{i}_{j}$,
        \item[(4)] $f^{i}_{j}\alpha\ro s_{R}(\alpha\rt|V^{j}_{k})=\alpha\rt f^{j}_{k} t_{R}(\alpha\ro|V^{i}_{j})$,
        \item[(5)] $\tilde{\sigma}(\lambda^{i}\ot_{A}\lambda^{j})=\lambda^{l}\bra t_{R}(f^{i}_{k}|V^{j}_{l})\ot_{A}\lambda^{k}$,
    \end{itemize}
    for any $\alpha, \beta\in \Omega$.
\end{prop}

\appendix
\section{Properties of Hopf algebroids} \label{sec6}


In this section, we will give some useful properties of (anti-)left and (anti-)right Hopf algebroids. Recall the shorthand for a (anti-)left Hopf algebroid
\begin{equation}X_{+}\ot_{B^{op}}X_{-}:=\lambda^{-1}(X\ot_{B}1)\end{equation}
\begin{equation}X_{[-]}\ot^{B^{op}}X_{[+]}:=\mu^{-1}(1\ot_{B}X)\end{equation}

We recall from \cite[Prop.~3.7]{schau1} that for a left Hopf algebroid,
\begin{align}
    \one{X_{+}}\ot_{B}\two{X_{+}}X_{-}&=X\ot_{B}1\label{equ. inverse lamda 1};\\
    \one{X}{}_{+}\ot_{B^{op}}\one{X}{}_{-}\two{X}&=X\ot_{B^{op}}1\label{equ. inverse lamda 2};\\
    (XY)_{+}\ot_{B^{op}}(XY)_{-}&=X_{+}Y_{+}\ot_{B^{op}}Y_{-}X_{-}\label{equ. inverse lamda 3};\\
    1_{+}\ot_{B^{op}}1_{-}&=1\ot_{B^{op}}1\label{equ. inverse lamda 4};\\
    \one{X_{+}}\ot_{B}\two{X_{+}}\ot_{B^{op}}X_{-}&=\one{X}\ot_{B}\two{X}{}_{+}\ot_{B^{op}}\two{X}{}_{-}\label{equ. inverse lamda 5};\\
    X_{+}\ot_{B^{op}}\one{X_{-}}\ot_{B}\two{X_{-}}&=X_{++}\ot_{B^{op}}X_{-}\ot_{B}X_{+-}\label{equ. inverse lamda 6};\\
    X&=X_{+}t_{L}(\varepsilon(X_{-}))\label{equ. inverse lamda 7};\\
    X_{+}X_{-}&=s_{L}(\varepsilon(X))\label{equ. inverse lamda 8}.
\end{align}
Since $\lambda(s_{L}(a)\ot_{B^{op}}1)=s_{L}(a)\ot_{B} 1$, we have
\begin{align*}
    &s_{L}(a)_{+}\ot_{B^{op}} s_{L}(a)_{-}=s_{L}(a)\ot_{B^{op}} 1\\
\end{align*}
As a result from (\ref{equ. inverse lamda 3}), we get
\begin{align}
    &(Xs_{L}(a))_{+}\ot_{B^{op}} (Xs_{L}(a))_{-}=X_{+}s_{L}(a)\ot_{B^{op}} X_{-}\label{equ. source and target map with lambda inv 1}\\
    &(s_{L}(a)X)_{+}\ot_{B^{op}} (s_{L}(a)X)_{-}=s_{L}(a)X_{+}\ot_{B^{op}} X_{-}\label{equ. source and target map with lambda inv 2}
\end{align}
Since $\lambda^{-1}(t_{L}(a)X\ot_{B}1)=\lambda^{-1}(X\ot_{B}s_{L}(a))=X_{+}\ot_{B^{op}}X_{-}s_{L}(a)$, we have
\begin{align}\label{equ. source and target map with lambda inv 3}
   (t_{L}(a)X)_{+}\ot_{B^{op}} (t_{L}(a)X)_{-}=X_{+}\ot_{B^{op}}X_{-}s_{L}(a).
\end{align}
Moreover, since $\lambda(X_{+}\ot_{B^{op}}s_{L}(a)X_{-})=\one{X_{+}}\ot_{B}\two{X_{+}}s_{L}(a)X_{-}=\one{X_{+}}t_{L}(a)\ot_{B}\two{X_{+}}X_{-}=Xt_{L}(a)\ot_{B}1$, we have
\begin{align}\label{equ. source and target map with lambda inv 4}
    (Xt_{L}(a))_{+}\ot_{B^{op}}(Xt_{L}(a))_{-}=X_{+}\ot_{B^{op}}s_{L}(a)X_{-}.
\end{align}
Finally, since $\lambda(t_{L}(a)X_{+}\ot_{B^{op}}X_{-})=\lambda(X_{+}\ot_{B^{op}}X_{-}t_{L}(a))=X\ot_{B}t_{L}(a)$, we have
\begin{align}\label{equ. source and target map with lambda inv 5}
    t_{L}(a)X_{+}\ot_{B^{op}}X_{-}=X_{+}\ot_{B^{op}}X_{-}t_{L}(a).
\end{align}

Similarly, for a anti-left Hopf algebroid, from \cite{BS} we have
\begin{align}
    \one{X_{[+]}}X_{[-]}\ot_{B}\two{X_{[+]}}&=1\ot_{B}X\label{equ. inverse mu 1};\\
    \two{X}{}_{[-]}\one{X}\ot^{B^{op}}\two{X}{}_{[+]}&=1\ot^{B^{op}}X\label{equ. inverse mu 2};\\
    (XY)_{[-]}\ot^{B^{op}}(XY)_{[+]}&=Y_{[-]}X_{[-]}\ot^{B^{op}}X_{[+]}Y_{[+]}\label{equ. inverse mu 3};\\
    1_{[-]}\ot^{B^{op}}1_{[+]}&=1\ot^{B^{op}}1\label{equ. inverse mu 4};\\
    X_{[-]}\ot^{B^{op}}\one{X_{[+]}}\ot_{B}\two{X_{[+]}}&=\one{X}{}_{[-]}\ot^{B^{op}}\one{X}{}_{[+]}\ot_{B}\two{X}\label{equ. inverse mu 5};\\
    (\one{X_{[-]}}\ot_{B}\two{X_{[-]}})\ot^{B^{op}}X_{[+]}&=(X_{[+][-]}\ot_{B}X_{[-]})\ot^{B^{op}}X_{[+][+]}\label{equ. inverse mu 6};\\
    X&=X_{[+]}s_{L}(\varepsilon(X_{[-]}))\label{equ. inverse mu 7};\\
    X_{[+]}X_{[-]}&=t_{L}(\varepsilon(X))\label{equ. inverse mu 8};\\
    (Xs_{L}(b))_{[-]}\ot^{B^{op}}(Xs_{L}(b))_{[+]}&=t_{L}(b)X_{[-]}\ot^{B^{op}} X_{[+]}\label{equ. inverse mu 9};\\
    (s_{L}(b)X)_{[-]}\ot^{B^{op}}(s_{L}(b)X)_{[+]}&=X_{[-]}t_{L}(b)\ot^{B^{op}} X_{[+]}\label{equ. inverse mu 10};\\
    (Xt_{L}(b))_{[-]}\ot^{B^{op}}(Xt_{L}(b))_{[+]}&=X_{[-]}\ot^{B^{op}} X_{[+]}t_{L}(b)\label{equ. inverse mu 11};\\
    (t_{L}(b)X)_{[-]}\ot^{B^{op}}(t_{L}(b)X)_{[+]}&=X_{[-]}\ot^{B^{op}} t_{L}(b)X_{[+]}\\
    X_{[-]}\ot^{B^{op}} s_{L}(b)X_{[+]}&=X_{[-]}s_{L}(b)\ot^{B^{op}} X_{[+]}\label{equ. inverse mu 12}.
\end{align}

Similarly, recall the shorthand for a (anti-)right Hopf algebroid $\cR$,
\begin{equation} X\m\ot_{A^{op}} X\p :=\hat{\lambda}^{-1}(1\ot_{A}X)\end{equation}
\begin{equation}X\np\ot^{A^{op}}X\nm:=\hat{\mu}^{-1}(X\ot_{A}1)\end{equation}

If $\cR$ is a right Hopf algebroid, we have
\begin{align}
    X\m X\p \ro \ot_{A}X\p \rt&=1\ot_{A} X \label{equ. inverse hlamda 1};\\
    X\ro X\rt \m\ot_{A^{op}}X\rt \p&=1\ot_{A^{op}}X\label{equ. inverse hlamda 2};\\
    (XY)\m\ot_{A^{op}}(XY)\p&=Y\m X\m\ot_{A^{op}}X\p Y\p\label{equ. inverse hlamda 3};\\
    1\m\ot_{A^{op}}1\p&=1\ot_{A^{op}}1\label{equ. inverse hlamda 4};\\
    X\m\ot_{A^{op}}X\p \ro\ot_{A}X\p \rt&=X\ro \m\ot_{A^{op}}X\ro \p\ot_{A}X\rt\label{equ. inverse hlamda 5};\\
    X\p \m\ot_{A}X\m\ot_{A^{op}}X\p \p&=X\m \ro\ot_{A}X\m \rt\ot_{A^{op}}X \p\label{equ. inverse hlamda 6};\\
    X&=t_{R}(\varepsilon(X\m))X\p \label{equ. inverse hlamda 7};\\
    X\m X\p&=s_{R}(\varepsilon(X))\label{equ. inverse hlamda 8};\\
    (Xs_{R}(a))\m\ot_{A^{op}} (Xs_{R}(a))\p&=X\m\ot_{B^{op}} X\p s_{R}(a)\label{equ. source and target map with hlambda inv 1};\\
    (s_{R}(a)X)\m\ot_{A^{op}} (s_{R}(a)X)\p&=X\m\ot_{B^{op}} s_{R}(a)X\p\label{equ. source and target map with hlambda inv 2};\\
   (t_{R}(a)X)\m\ot_{A^{op}} (t_{R}(a)X)\p&=X\m S_{R}(a)\ot_{A^{op}}X\p\label{equ. source and target map with hlambda inv 3};\\
    (Xt_{R}(a))\m\ot_{A^{op}} (Xt_{R}(a))\p&=s_{R}(a)X\m \ot_{A^{op}}X\p\label{equ. source and target map with hlambda inv 4};\\
    t_{R}(a)X\m\ot_{A^{op}}X\p&=X\m\ot_{A^{op}}X\p t_{R}(a)\label{equ. source and target map with hlambda inv 5}.
\end{align}

Similarly, for a anti-right Hopf algebroid,
\begin{align}
    X\np\ro\ot_{A}X\nm X\np\rt&=X\ot_{A}1\label{equ. inverse hmu 1};\\
    X\ro \np\ot^{A^{op}}X\rt X\ro \nm&=X\ot^{A^{op}}1\label{equ. inverse hmu 2};\\
    (XY)\np\ot^{A^{op}}(XY)\nm&=X\np Y\np\ot^{A^{op}}Y\nm X\nm\label{equ. inverse hmu 3};\\
    1\np\ot^{A^{op}}1\nm&=1\ot^{A^{op}}1\label{equ. inverse hmu 4};\\
    X\np \ro\ot_{A}X\np \rt\ot^{A^{op}} X\nm &=X\ro\ot_{A}X\rt \np\ot^{A^{op}} X\rt \nm\label{equ. inverse hmu 5};\\
    X\np\ot^{A^{op}}X\nm \ro\ot_{A} X\nm \rt&=X\np \np\ot^{A^{op}}X\nm\ot_{A}X\np \nm\label{equ. inverse hmu 6};\\
    X&=s_{R}(\varepsilon(X\nm))X\np\label{equ. inverse hmu 7};\\
    X\nm X\np&=t_{R}(\varepsilon(X))\label{equ. inverse hmu 8};\\
    (Xs_{R}(b))\np\ot^{A^{op}}(Xs_{R}(b))\nm&=X\np\ot^{A^{op}} t_{R}(b)X\nm\label{equ. inverse hmu 9};\\
    (s_{R}(b)X)\np\ot^{A^{op}}(s_{R}(b)X)\nm&=X\np\ot^{A^{op}} X\nm t_{R}(b)\label{equ. inverse hmu 10};\\
    (Xt_{R}(b))\np\ot^{A^{op}}(Xt_{R}(b))\nm&=X\np t_{R}(b)\ot^{A^{op}} X\nm\label{equ. inverse hmu 11};\\
    (t_{R}(b)X)\np\ot^{A^{op}}(t_{R}(b)X)\nm&=t_{R}(b)X\np\ot^{A^{op}} X\nm\\
    X\np\ot^{A^{op}} s_{R}(b)X\nm&= X\np  s_{R}(b)\ot^{A^{op}}X\nm\label{equ. inverse hmu 12},
\end{align}

\end{document}